\documentclass[12pt]{amsart}
\usepackage[margin=1in,includeheadfoot]{geometry}
\usepackage{amsmath}
\usepackage{slashed}
\usepackage{graphicx}

\usepackage[margin=1in,includeheadfoot]{geometry}
\usepackage{amsmath}
\usepackage{slashed}
\usepackage{graphicx}
\usepackage{mathrsfs}
\usepackage{txfonts}

\usepackage{amssymb,dsfont}
\usepackage{enumerate}
\usepackage{slashed}
\usepackage{fancyhdr}
\usepackage{aliascnt}
\usepackage{pdfsync}
\usepackage{hyperref}

\newtheorem{thm}{Theorem}[section]
\newtheorem{cor}[thm]{Corollary}
\newtheorem{lem}[thm]{Lemma}
\newtheorem{prop}[thm]{Proposition}
\theoremstyle{definition}
\newtheorem{defn}[thm]{Definition}
\theoremstyle{remark}
\newtheorem{rem}[thm]{Remark}
\numberwithin{equation}{section}
\newcommand{\norm}[1]{\left\Vert#1\right\Vert}

\newcounter{stepnum}

\def\bee{\begin{eqnarray}}
\def\beee{\begin{eqnarray*}}
\def\eee{\end{eqnarray}}
\def\eeee{\end{eqnarray*}}

\def\ba{\begin{array}}
\def\ea{\end{array}}

\def\R{\mathbb R}

\newcommand{\M}{M}

\makeatletter

\newcommand{\Rmnum}[1]{\expandafter\@slowromancap\romannumeral #1@}
\makeatother

\begin{document}

\title[Super-Liouville equation with a spinorial Yamabe type  term]{Super-Liouville equation with a spinorial Yamabe type  term}

\author[Liu]{Lei Liu}
\address{School of Mathematics and Statistics, Key Laboratory of Nonlinear Analysis and Applications (Ministry of Education), Hubei Key Laboratory of Mathematical Sciences, Central China Normal
University, Wuhan, 430079, People's Republic of China}%
\email{leiliu2020@ccnu.edu.cn}

\author[Wei]{Mingjun Wei}
\address{School of Mathematics and Statistics,  Central China Normal
University, Wuhan, 430079, People's Republic of China}
\email{mingjunw@mails.ccnu.edu.cn}

\thanks{}

\subjclass[2010]{}

\date{\today}
\begin{abstract}
In this paper, we study the super-Liouville equation with a spinorial Yamabe type  term, a natural generalization of Liouville equation, super-Liouville equation and spinorial Yamabe type equation. We establish some refined qualitative properties for such a blow-up sequence. In particular, we show energy identities not only for  the spinor part but also for the function part. Moreover, the local masses at a blow-up point are also computed. A new phenomenon is that there are two kinds of singularities and local masses due to the nonlinear spinorial Yamabe type term, which is different from super-Liouville equation.
\end{abstract}
\maketitle
\section{Introduction}

The compactness of the solution space of a nonlinear partial differential equation plays an important role in the study of the corresponding existence problem. In dimension two, many interesting variational problems in geometry and physics are borderline cases of the Palais-Smale condition and thus one can not use variational methods directly to get the existence result, such as harmonic maps, minimal and prescribed mean curvature surfaces in Riemannian manifold, Liouville type equation, Ginzburg-Landau type problems and so on.

The classical Liouville functional for a real-valued function $u$ on a Riemann surface $(M,g)$ is $$L_1(u)=\int_M\left(\frac{1}{2}|\nabla u|^2+K_gu-e^{2u}\right)dM,$$ where $K_g$ is the Gauss curvature of $M$. The Euler-Lagrange equation for $L_1$ is the following Liouville equation $$-\Delta_g u=2e^{2u}-K_g,$$ which plays a fundamental role in many two dimensional physical models, complex analysis and differential geometry of Riemann surface, in particular in the problem of prescribed Gaussian curvature. The blow-up analysis for Liouville equation was  systematically developed by Brezis-Merle \cite{Brezis},  Li-Shafrir \cite{Li-Shafrir} and Li \cite{Li} etc.

In physics, Liouville type equation also occurs naturally in string theory as discovered by Polyakov \cite{Polyakov}, from the gauge anomaly in quantizing the string action. There is also a natural supersymmetric version of the Liouville functional, coupling the bosonic scalar field to a fermionic spinor field. Motivated from supersymmetric string theory, Jost-Wang-Zhou \cite{Jost-Wang-Zhou} introduced the following super-Liouville functional for a real-valued function $u$ and a complex valued spinor $\psi$ $$L_2(u,\psi)=\int_M\left(\frac{1}{2}|\nabla u|^2+K_gu+\langle (\slashed{D}_g+e^u)\psi,\psi \rangle_{\Sigma M}-e^{2u}\right)dM,$$ which is conformally invariant.
Here, $(M,g)$ is a Riemann surface with a fixed spin structure, $\Sigma M$  the spinor bundle over $M$ and $\langle\cdot,\cdot\rangle_{\Sigma M}$  the metric on $\Sigma M$. Choosing a local orthonormal basis ${e_\alpha,\alpha=1,2}$ on $M$, the  Dirac operator is
defined as $\slashed{D}_g:=e_\alpha\cdot\nabla_{e_\alpha}$, where $\nabla$ is the spin connection on $\Sigma M$ and $\cdot$ is the Clifford multiplication. This multiplication is  skew-adjoint:
\[
\langle X\cdot\psi,\varphi \rangle_{\Sigma M}=-\langle \psi,X\cdot\varphi \rangle_{\Sigma M}
\]
for any $X\in\Gamma(TM)$, $\psi,\ \varphi\in\Gamma(\Sigma M)$.

The Euler-Lagrange equation for $L_2$ is the following super-Liouville equation
\begin{align*}
\begin{cases}
-\Delta_g u&=2e^{2u}-e^u\langle \psi,\psi\rangle-K_g,\\
-\slashed{D}_g\psi &=e^u\psi,
\end{cases} \ \ in \ \ M.
\end{align*}
Similarly to Liouville equation, super-Liouville equation is also conformally invariant, which implies that the solution space is in general not compact. The blow-up analysis for super-Liouville equation was systematically developed in \cite{Jost-Wang-Zhou,Jost-Wang-Zhou-Zhu-1, Jost-Zhou-Zhu} etc.

In two dimension, there is a nonlinear Dirac-type system, the so called spinorial Yamabe type equation which is defined as follows:
\begin{equation}\label{equat:02}
\slashed{D}_g\psi=H(x)|\psi|^2\psi\ \ in\ \ M,
\end{equation} where $|\psi|^2=\langle\psi,\psi\rangle_{\Sigma M}$, $\psi\in \Gamma(\Sigma M)$. This type of nonlinear Dirac equations appears naturally in  geometry and physics. Firstly, by spinorial Weierstrass representation, the solution of \eqref{equat:02} can be used to find the existence of prescribed mean curvature surfaces in $\R^3$ (here the function $H$ plays the role of the mean curvature). See for instance \cite{Taimanov}. Secondly, consider the Dirac-harmonic map with curvature term introduced by Chen-Jost-Wang \cite{Chen-Jost-Wang-1,Chen-Jost-Wang-2},  which was derived from the nonlinear supersymmetric $\sigma$-model of quantum field theory,  then the nonlinear Dirac equation for the spinor reduces to \eqref{equat:02} when the map is a constant. The blow-up analysis for \eqref{equat:02} was studied by \cite{Zhu,Chen-Liu-Zhu}.

\

In this paper, we consider the following functional
\begin{equation}\label{equat:functional}
L(u,\psi):=\int_M\left(\frac{1}{2}|\nabla u|^2+K_gu+\langle (\slashed{D}_g+e^u)\psi,\psi \rangle_{\Sigma M}+F(x)|\psi|^4-e^{2u}\right)dM,
\end{equation} where $F(x): M\to\R$ is a $C^1$ function. The Euler-Lagrange equation for $L$ is the following system
\begin{align}\label{equat:01}
\begin{cases}
-\Delta u&=2e^{2u}-e^u\langle\psi,\psi\rangle-K_g,\\
\slashed{D}_g\psi&=-e^u\psi-2F|\psi|^2\psi,
\end{cases} \ \ in\ \ M.
\end{align}

We call \eqref{equat:01} the \emph{super-Liouville equation with a spinorial Yamabe type term}. It is easy to see that the special forms of \eqref{equat:01} include Liouville equation ($\psi\equiv 0$), super-Liouville equation ($F\equiv 0$) and the spinorial Yamabe-type equation \eqref{equat:02} (when the function $u$ vanishes). The important point is that this generalization preserves a fundamental property of the energy functional on Riemann surfaces, i.e. the conformal invariance.

\

In present paper, we will provide an analytic properties for super-Liouville equation with a spinorial Yamabe type term, such as removable singularity, Brezis-Merle type concentration phenomenon, energy identity, local masses and so on.

Denote $$E(u,\psi;\Omega):=\int_{\Omega}\left(e^{2u}+|\psi|^4\right)dM,\ \ \Omega\subset M.$$
Now we state our  first main result which is about the following Brezis-Merle type concentration phenomenon.

\begin{thm}\label{thm:main-1}
Let $(u_n,\psi_n)$ be a sequence of smooth solutions of \eqref{equat:01}  with bounded energy $$E(u_n,\psi_n;M)\leq C.$$ Define
\begin{align*}
\Sigma_1&:=\{x\in M|\ \mbox{there is a sequence of points } x_n\to x\mbox{ such that } u_n(x_n)\to+\infty\},\\
\Sigma_2&:=\{x\in M|\ \mbox{there is a sequence of points } x_n\to x\mbox{ such that } |\psi_n(x_n)|\to+\infty\}.
\end{align*} Then $\Sigma=\Sigma_1\cup\Sigma_2$ is a finite point set. Passing to a subsequence, we have
\begin{itemize}
  \item[(1)] $\psi_n$ is bounded in $L^\infty_{loc}(M\setminus \Sigma_2)$.

  \

  \item[(2)] For $u_n$, one of the following alternatives holds:
  \begin{itemize}
    \item[(a)] $u_n$ is bounded in $L^\infty(M)$.

    \

    \item[(b)] $u_n\to -\infty$ uniformly on $M$.

    \

    \item[(c)] $\Sigma$ is nonempty and either $u_n$ is bounded in $L^\infty_{loc}(M\setminus \Sigma)$ or $u_n\to -\infty$ uniformly on any compact subsets of $M\setminus \Sigma$.
  \end{itemize}
\end{itemize}
\end{thm}

\

Different from super-Liouville equation \cite{Jost-Wang-Zhou} where $\Sigma_2\subset \Sigma_1$, super-Liouville equation with a spinorial Yamabe type term will no longer possess such property. This is why we define  two kinds of singularities below (see Definition \ref{def:01}).

\

Next we want to construct the full blow up theory for \eqref{equat:01}. Similar to Liouville equation and super-Liouville system, this should contain the energy identity for solutions, the blow-up values, and the profile of
solutions near the blow-up point.

For this aim, we first  establish the following energy identity for the spinor.
\begin{thm}\label{thm:main-2}
  Let $M$ be a closed Riemannian surface with a fixed spin structure. Suppose $(u_n,\psi_n)$ is a sequence of smooth solution to \eqref{equat:01} with bounded energy $  E(u_n,\psi_n;M)<C$.
  Suppose $\Sigma\neq \emptyset$ and set $\Sigma:=\{x_1,\cdots,x_l\}$. Then there are finitely many solutions $(u^{i,k},\psi^{i,k}),\ \ i=1,\cdots,l;\ k=1,\cdots,L_i$ of
  \begin{equation}\label{equat:10}
    \begin{cases}
      -\Delta u^{i,k} &= 2e^{2u^{i,k}} - e^{u^{i,k}} |\psi^{i,k}|^2-K_{g_{S^2}} ,\\
      \slashed{D}\psi^{i,k} &= -e^{u^{i,k}}\psi^{i,k} - 2\mu_i|\psi^{i,k}|^2\psi^{i,k},
  \end{cases}\ \ on\ \  S^2,
\end{equation}
   where $\mu_i=F(x_i)$ is a constant, $(S^2,g_{S^2})$ is the standard two dimensional sphere, $K_{g_{S^2}}=1$ is the Gauss curvature, such that
  passing to a subsequence, $\psi_n$ converges in $C_{loc}^{\infty}$ to $\psi$ on $M\setminus \Sigma$.

  Moreover, we have the energy identity
  \begin{equation*}
    \lim_{n\rightarrow \infty} \int_{M} |\psi_n|^4dM = \int_{M}|\psi|^4dM + \sum_{i=1}^{l} \sum_{k=1}^{L_i} \int_{S^2} |\psi^{i,k}|^4dx.
  \end{equation*}

\end{thm}

\

\begin{rem}
In our paper,  we call $(u,\psi)$ is a bubble if it satisfies equation \eqref{equat:10}. We want to remark that for bubble solution \eqref{equat:10}, if $u=-\infty$, i.e. the function part vanishes, it reduces to a spinorial Yamabe type equation $$\slashed{D}\psi=-2\mu |\psi|^2\psi\ \ on\ \ S^2.$$ Then we call $\psi$ \emph{a spinorial Yamabe type bubble}.  Otherwise, if the function part does not vanish, we call $(u,\psi)$ \emph{a super-Liouville type bubble}.
\end{rem}

\

For super-Liouville equation with a spinorial Yamabe type term, Theorem \ref{thm:main-1} tells us that there may appear two kinds of bubbles at a same blow-up point, i.e. super-Liouville type bubble and spinorial Yamabe type bubble, which makes the blow-up phenomenon and related analysis different and in fact more complicated than Liouville equation and super-Liouville equation. In order to characterise the refined quantitative properties, we need to define two kinds of singularities for the blow-up set $\Sigma$.

\

Since $\Sigma$ has finite points, we can take $r_0>0$ small such that $B_{r_0}(x_i),\ x_i\in \Sigma$ are mutually disjointed.
\begin{defn}\label{def:01}
$p\in\Sigma$ is called a first type singular point for $(u_n,\psi_n)$ if there exists a universal constant $C>0$ such that $$\liminf_{n\to\infty}\big(\max_{x\in\overline{B_{r_0}(p)}}u_n(x)-2\ln (1+\max_{x\in\overline{B_{r_0}(p)}}|\psi_n(x)|)\big)\geq -C.$$ Otherwise, we call $p\in\Sigma$ a second type singular point.
\end{defn}

We want to remark that for super-Liouville equations \cite{Jost-Wang-Zhou}, the second type singularity will not appear.

\

Denote
\begin{align*}
\Sigma_3:=\bigg\{p\in\Sigma\ |\ \ p\ \ \mbox{ is a first type singular point} \bigg\}.
\end{align*}

 With the help of Spinor's energy identity Theorem \ref{thm:main-2}, if the blow-up set $\Sigma_3\neq \emptyset$, then we can rule out the first case in $(c)$ of Theorem \ref{thm:main-1}. We have the following result that
\begin{thm}\label{thm:main-3}
   Assume that $(u_n,\psi_n)$ is a sequence of solution to \eqref{equat:01} with finite energy condition $E(u_n,\psi_n;M)<C$ and the blow-up set $\Sigma_3\neq \emptyset$, then
  \begin{equation*}
    u_n \rightarrow -\infty\ \ uniformly\ \ on \  \ compact\ \ subset\ \ of \ \ M \setminus \Sigma
  \end{equation*}and
  \begin{equation*}
    2e^{2u_n} - e^{u_n}|\psi_n|^2  \longrightarrow \sum_{x_i\in \Sigma_3} \alpha_i \delta_{x_i}+\sum_{x_i\in \Sigma\setminus \Sigma_3} \beta_i \delta_{x_i}
  \end{equation*}
  in the sense of measure where $\alpha_i\ge 4\pi$.
\end{thm}

\

Defining the blow-up value (or local mass) at a blow-up point $p\in\Sigma$ as
\begin{equation*}
  m(p):=\lim_{r\rightarrow 0}\lim_{n\rightarrow 0} \int_{B_r(p)} \left(2e^{2u_n} - e^{u_n}|\psi_n|^2\right)dx,
\end{equation*}
different from Liouville equation and super-Liouville equation, we have two kinds of blow-up values due to the nonlinear spinor Yamabe term.
\begin{thm}\label{thm:main-4}
 Suppose $\Sigma_3\neq \emptyset$. If $p\in \Sigma_3$, then $ m(p)= 4\pi.$ If $p=x_i\in \Sigma\setminus \Sigma_3$, then $m(p)=0\ \ or\ \ 4\pi.$
\end{thm}

\

Next we  give a more clear picture for the blow-up analysis in Theorem \ref{thm:main-2}. We also show the energy identity for the function part $u_n$. Precisely, our result is
\begin{thm}\label{thm:main-5}
Under assumptions  and notations of Theorem \ref{thm:main-2}, suppose $\Sigma_3\neq \emptyset$. Denote $$\tilde{\Sigma}_3:=\{p\in\Sigma\ |\ p\in\Sigma_3\ \ or\ \ p\in \Sigma\setminus \Sigma_3 \ \ with\ \  m(p)=4\pi \}.$$ Then for each $x_i\in \tilde{\Sigma}_3$, among the bubbles $(u^{i,k},\psi^{i,k})$, $k=1,...,L_i$ in Theorem \ref{thm:main-2}, there is only one super-Liouville type bubble (w.l.o.g, denoted by $u^{i,1}$) where the other bubbles are all spinorial Yamabe type bubble.

Moreover, we have following energy identity for the function that
 \begin{equation*}
    \lim_{n\rightarrow \infty} \int_{M} e^{2u_n}dM = \sum_{x_i\in \tilde{\Sigma}_3}  \int_{S^2} e^{2u^{i,1}}dx.
  \end{equation*}
\end{thm}

\

Lastly, we discuss the case that $\Sigma_3=\emptyset$, i.e. there is no first type singularity.
\begin{thm}\label{thm:main-6}
Under assumptions  and notations of Theorem \ref{thm:main-2}, suppose $\Sigma_3= \emptyset$. Then one of the following two alternatives holds.
\begin{itemize}
\item[(1)]  $u_n$ is uniformly bounded in $L^\infty_{loc}(M\setminus \Sigma)$, then passing to a subsequence, we have

\

 \begin{itemize}
 \item[(i)] $(u_n,\psi_n)\to (u,\psi)$ in $C^2_{loc}(M\setminus \Sigma)$ where $(u,\psi)$ satisfies super-Liouville equation with a spinorial Yamabe type term \eqref{equat:01} in $M$.

     \

  \item[(ii)] For any $p\in\Sigma$, there holds $m(p)=0$. This also implies that the bubbles $(u^{i,k},\psi^{i,k})$, $i=1,...,l$, $k=1,...,L_i$ in Theorem \ref{thm:main-2}, are all spinorial Yamabe type bubble.

      \

  \item[(iii)] Energy identity for the function
 \begin{equation*}
    \lim_{n\rightarrow \infty} \int_{M} e^{2u_n}dM =   \int_{M} e^{2u}dx.
  \end{equation*}
 \end{itemize}

 \

 \item[(2)]  Passing to a subsequence, $u_n\to -\infty$  uniformly on any compact subset of $M\setminus \Sigma$. Then we have
 \begin{itemize}
 \item[(i)] For $p\in\Sigma$, there holds $m(p)=0$ or $4\pi$.

 \

  \item[(ii)] Denote $$\tilde{\Sigma}_3:=\{p\in\Sigma\ |\ \ p\in \Sigma \ \ with\ \  m(p)=4\pi \}.$$ Then for each $x_i\in \tilde{\Sigma}_3$, among the bubbles $(u^{i,k},\psi^{i,k})$, $k=1,...,L_i$ in Theorem \ref{thm:main-2}, there is only one super-Liouville type bubble (w.l.o.g, denoted by $u^{i,1}$) where the other bubbles are all spinorial Yamabe type bubble.

\

\item[(iii)] Energy identity for the function
 \begin{equation*}
    \lim_{n\rightarrow \infty} \int_{M} e^{2u_n}dM = \sum_{x_i\in \tilde{\Sigma}_3}  \int_{S^2} e^{2u^{i,1}}dx.
  \end{equation*}

 \end{itemize}
\end{itemize}
\end{thm}

\

At the end of this section, we want to remark that super-Liouville equation with a spinorial Yamabe type term is not just a simple promotion of super-Liouville equation. Comparing to Liouville equation or super-Liouville equation, the nonlinear spinorial Yamabe term in \eqref{equat:01} will produce some different blow-up phenomenons which of course makes the relevant neck analysis more complicated, such as:
\begin{itemize}
\item[(1)] For super-Liouville equation, there is only first type singularity, while for super-Liouville equation with a spinorial Yamabe type term, there are two kinds of singularities (see Definition \ref{def:01}). One can see that if we make $u$ disappear, i.e. $u=-\infty$, it reduces to spinorial Yamabe type equation \eqref{equat:02} and then all the singularities are of second type. Different singularities also generate different blow-up values. See Theorem \ref{thm:main-4}.

\

\item[(2)] From the work \cite{Jost-Wang-Zhou,Jost-Wang-Zhou-Zhu-1}, it is now not hard to conclude that for super-Liouville equation, there is just one bubble at each blow-up point which is a super-Liouville type bubble. However, this property will no longer hold for super-Liouville equation with a spinorial Yamabe type term (see Theorem \ref{thm:main-2}), where there may appear finitely many bubbles and meanwhile contain two types.

\

\item[(3)] To prove the energy identity for the function part, i.e. estimating the function's integration $\int_{A_n} e^{2u_n}dx=o(1)$ where $A_n$ is a neck domain, the key idea is to transfer the problem to the estimate of spinor and then we can apply the spinor's energy identity. Precisely, using equation \eqref{equat:01} and integrating by parts, with the help of blow-up values and the classification of super-Liouville type bubble, we first show that $\int_{A_n} (2e^{2u_n}-e^{u_n}|\psi_n|^2)dx=o(1)$. Then we just need to prove   $\int_{A_n} e^{u_n}|\psi_n|^2dx=o(1)$. For this, if there is only one bubble (e.g. for super-Liouville equation), then it is easy to see that its limit is $0$ which is a direct consequence of the energy identity for spinor, i.e. $\int_{A_n} |\psi_n|^4dx=o(1)$. However, here we need to deal with the case of multiple bubbles, which is the main difficulty and contribution of this paper. It needs more detailed and delicated analysis. See Lemma \ref{lem:05} and Lemma \ref{lem:06}.
\end{itemize}

\

The rest of the paper is organized as follows. In Section \ref{sec:Some basic lemmas and Brezis-Merle type compactness}, we will firstly establish some basic lemmas including  small  energy regularity lemma, Pohozaev type identity. Secondly, we will prove a  Brezis-Merle type concentration property, i.e. Theorem \ref{thm:main-1}. In Section \ref{sec:removability of local singularities}, we will provide a necessary and sufficient condition for removing a local singularity. In Section \ref{sec:Bubble's equation}, we will show some classification properties for super-Liouville type bubble, such as the asymptotic behavior near infinity, the bubble's energy and so on. We will prove the energy identity for the spinor in Section \ref{sec:energy-identity-spinor}, i.e. Theorem \ref{thm:main-2}. The proofs of Theorem \ref{thm:main-3}-Thorem \ref{thm:main-6} will be given in Section \ref{sec:Blow-up values and energy identity for the function}.

\

\section{Some basic lemmas and Brezis-Merle type compactness}\label{sec:Some basic lemmas and Brezis-Merle type compactness}

\

In this section,  on one hand, we will  establish some basic lemmas  in the blow-up analysis for super-Liouville equation with a Yamabe term, such as a small energy regularity lemma, Pohozaev type identity and so on.  On the other hand, we will prove a Brezis-Merle type concentration property, i.e. Theorem \ref{thm:main-1}.

\

We start this part by proving following small energy regularity lemma.
\begin{lem}\label{lem:small-energy-regul}
Let $(u_n,\psi_n)$ be a sequence of smooth solutions of
\begin{equation}\label{equat:07}
  \begin{cases}
    -\Delta u_n &= 2e^{2u_n} - e^{u_n} |\psi_n|^2 + H^1_n(x) |\psi_n|^4,\\
    \slashed{D}\psi_n &= -e^{u_n}\psi_n +H^2_n(x)|\psi_n|^2\psi_n,
\end{cases} \ \ in\ \  B_1(0),
\end{equation} where $\sum_{i=1}^2\|H^i_n\|_{L^\infty(B_1(0))}\leq C$.  Then there exists a positive constant $\epsilon_1$, such that if $E(u_n,\psi_n;B_1)\leq\epsilon_1$, then the following alternatives hold:
\begin{itemize}
\item[(1)] $\|u_n^+\|_{L^\infty(B_{\frac{1}{2}})}+\|\psi_n\|_{L^\infty}(B_{\frac{1}{2}})\leq C.$

\

\item[(2)] Either $\|u_n\|_{L^\infty(B_{\frac{1}{2}})}\leq C$ or $u_n$ converges to $-\infty$ uniformly in $ B_{\frac{1}{2}}(0)$.

    \

\item[(3)] If we additionally assume $$osc_{\partial B_1(0)}u_n\leq C,\ \ \|\nabla H^2_n\|_{L^\infty(B_1(0))}\leq C,$$ then we have $$osc_{B_{\frac{1}{2}}(0)}u_n\leq C,\ \  \|\psi_n\|_{L^\infty(B_{\frac{1}{2}}(0))}+\|\nabla\psi_n\|_{L^\infty(B_{\frac{1}{2}}(0))}\leq C\|\psi_n\|_{L^4(B_{1}(0))}.$$
\end{itemize}
\end{lem}
\begin{proof}
\noindent\textbf{Step 1.} Estimate for $\|\psi_n\|_{L^{16}(B_{\frac{3}{4}})}$.

\

Take a cut-off function $\eta\in C^\infty_0(B_1)$ such that $0\leq\eta\leq 1$, $\eta|_{B_{\frac{3}{4}}}\equiv 1$ and $\|\nabla\eta\|_{L^\infty}\leq C$. From the equation of spinor \eqref{equat:01} and standard elliptic estimates of Dirac operator, for any $1<p<2$, we have
\begin{align*}
  \|\eta\psi_n\|_{W^{1,p}(B_{1})}&\leq C(\|\slashed{D}(\eta\psi_n)\|_{L^p(B_1)}+\|\eta\psi_n\|_{L^4(B_1)})\\
  &\leq C(\|e^{u_n}\eta\psi_n\|_{L^p(B_1)}+\|\ |\psi_n|^2|\eta\psi_n|\ \|_{L^p(B_1)}+\||\nabla\eta||\psi_n|\|_{L^p(B_1)}+\|\psi_n\|_{L^4(B_1)})\\
  &\leq C(\|e^{u_n}\|_{L^2(B_1)}+\|\psi_n \|^2_{L^4(B_1)})\|\eta\psi_n\|_{L^{\frac{2p}{2-p}}(B_1)}+C\|\psi_n\|_{L^4(B_1)}\\
  &\leq C\sqrt{\epsilon_1}\|\eta\psi_n\|_{W^{1,p}(B_{1})}+C\|\psi_n\|_{L^4(B_1)},
\end{align*} where we used Sobolev embedding theory and Young's inequality. Now, take $p=\frac{16}{9}$ and $\epsilon_1$ small enough, we get $\|\eta\psi_n\|_{W^{1,p}(B_{1})}\leq  C\|\psi_n\|_{L^4(B_1)},$ which implies $\|\psi_n\|_{L^{16}(B_{\frac{3}{4}})}\leq C\|\psi_n\|_{L^4(B_1)}.$

\

\noindent\textbf{Step 2.} Estimate for $\|u^+_n\|_{L^\infty(B_{\frac{1}{2}})}$ and $\|\psi_n\|_{L^\infty(B_{\frac{1}{2}})}$ .

\

Let $u_n^1$ be the solution of the following Dirichlet problem
\begin{align*}
\begin{cases}
  -\Delta u_n^1=2e^{2u_n}-e^{u_n}|\psi_n|^2+H_n^1(x)|\psi_n|^4,\ \ &in\ \ B_1,\\
  u^1_n=0,\ \ &on \ \ \partial B_1.
  \end{cases}
\end{align*}

Since $\|\Delta u_n^1\|_{L^1(B_1)}\leq C\epsilon_1,$ taking $\epsilon_1$ small, by Theorem 1 in \cite{Brezis}, we have $\int_{B_1}e^{8|u_n^1|}dx\leq C.$

Now, let $u_n^2=u_n-u_n^1$. Then it is easy to see that $u_n^2$ is a harmonic function in $B_1$. By mean value property, we have $$\|(u_n^2)^+\|_{L^\infty(B_{\frac{3}{4}})}\leq C\|(u_n^2)^+\|_{L^1(B_1)}\leq C(\|u_n^+\|_{L^1(B_1)}+\|u_n^1\|_{L^1(B_1)}) \leq C,$$ where we used the fact that $\|u_n^+\|_{L^1(B_1)}\leq \|e^{u_n}\|_{L^1(B_1)}.$

Then it is easy to see that $\|\Delta u_n^1\|_{L^2(B_{\frac{3}{4}})}\leq C$ which implies $\|u_n^1\|_{C^{0}(B_{\frac{3}{4}})}\leq C$. Thus, we arrived at $$\|u_n^+\|_{L^\infty(B_{\frac{3}{4}})}\leq C.$$

 Combining this with the fact $\|\psi_n\|_{L^{16}(B_{\frac{3}{4}})}\leq C$, by the standard elliptic estimates of Dirac operator and Sobolev embedding, we get $$\|\psi_n\|_{W^{1,4}(B_{\frac{5}{8}})}\leq C(\|\slashed{D}\psi_n\|_{L^{4}(B_{\frac{3}{4}})}+\|\psi_n\|_{L^4(B_{\frac{3}{4}})})\leq C,$$ and $\|\psi_n\|_{C^0(B_{\frac{5}{8}})}\leq C.$

\

\noindent\textbf{Step 3.} Conclusions of the lemma.

\

Since $\Delta (u_n-u^1_n)=0$ in $B_1(0)$ and $u_n-u^1_n$ is bounded from above in $B_{\frac{3}{4}}(0)$, then the second conclusion follows from Harnack's inequality.

For the third conclusion, since $\Delta (u_n-u^1_n)=0$ in $B_1(0)$ and $osc_{\partial B}(u_n-u^1_n)\leq C$, by maximal principle property, we know $osc_{B_{\frac{3}{4}}(0)}u_n\leq C.$

To estimate $\|\nabla \psi_n\|_{L^\infty(B_{\frac{1}{2}}(0))}$, by above estimates, we now have that $$\|\Delta (u_n(x)-\overline{u}_n)\|_{L^\infty(B_{\frac{5}{8}}(0))}\leq C,\ \ i=1,2,$$ where $\overline{u}_n:=\frac{1}{|B_{\frac{5}{8}}(0)|}\int_{B_{\frac{5}{8}}(0)}u_n(x)dx$ . The standard elliptic theory and embedding theory yields that
\begin{align*}
\|\nabla u_n(x)\|_{L^{\infty}(B_{\frac{9}{16}}(0))}&\leq\|u_n(x)-\overline{u}_n\|_{W^{2,4}(B_{\frac{9}{16}}(0))}\\ &\leq C\left(\|\Delta (u_n(x)-\overline{u}_n)\|_{L^4(B_{\frac{5}{8}}(0))}+ \|u_n(x)-\overline{u}_n\|_{L^4(B_{\frac{5}{8}}(0))}\right)\\
&\leq C\left(\|\Delta (u_n(x)-\overline{u}_n)\|_{L^\infty(B_{\frac{5}{8}}(0))}+ osc_{B_{\frac{5}{8}}(0)}u_n(x)\right)\leq C.
\end{align*}
Using the Schr\"{o}dinger-Lichnerowicz formula \cite{lawson1989spin}, we have
$$-\Delta\psi_n=\slashed{D}^2\psi_n=-\slashed{D}\left(e^{u_n}\psi_n+H^2_n(x)e^{u_n}\psi_n\right).$$ Then $$|\Delta\psi_n(x)|\leq C\left(|\psi_n(x)| +|\nabla\psi_n(x)|\right),\ \ \forall\ x\in B_{\frac{9}{16}}(0)$$ and the elliptic estimate tells us that $$\|\psi_n\|_{W^{2,4}(B_{\frac{1}{2}}(0))}\leq C(\|\psi_n\|_{L^{4}(B_{\frac{5}{8}}(0))}+\|\nabla\psi_n\|_{L^{4}(B_{\frac{5}{8}}(0))})\leq C\|\psi_n\|_{L^{4}(B_{\frac{3}{4}}(0))},$$ which immediately implies $\|\nabla\psi_n\|_{L^\infty(B_{\frac{1}{2}}(0))}\leq C\|\psi_n\|_{L^{4}(B_{\frac{3}{4}}(0))}$.

\end{proof}

\

With the help of Lemma \ref{lem:small-energy-regul}, we can prove the following Brezis-Merle type compactness.

\begin{proof}[\textbf{Proof of Theorem \ref{thm:main-1}:}]

We firstly see that $\psi_n$ is bounded in $L^\infty_{loc}(\M\setminus \Sigma)$ by the definition of $\Sigma$. Define $$S:=\{x\in M|\ \liminf_{n\to\infty}E(u_n,\psi_n;B_r(x))\geq\epsilon_1,\  \ \forall r>0\},$$ where $\epsilon_1$ is the constant in Lemma \ref{lem:small-energy-regul}. Since the total energy is bounded, it is easy to see that $S$ is a finite point set. We divide the proof into two steps.

\

\noindent\textbf{Step 1.} $\Sigma=S$.

\

On the one hand,  we prove $\Sigma\subset S$, i.e. for any $x_0\in \Sigma$, there holds $x_0\in S$. In fact, if not, then there exists a constant $r_0>0$, such that passing to a subsequence, there holds
$$\lim_{n\to\infty}E(u_n,\psi_n;B_{r_0}(x_0))<\epsilon_1.$$ By Lemma \ref{lem:small-energy-regul}, we have $$\|u_n^+\|_{L^\infty(B_{\frac{r_0}{2}}(x_0))}+\|\psi_n\|_{L^\infty(B_{\frac{r_0}{2}}(x_0))}\leq C,$$ which contradicts to $x_0\in \Sigma$.

On the other hand, we show $S\subset \Sigma $, i.e. for any $x_0\in S$, it holds $x_0\in \Sigma$. We first claim that for any $r>0$, $$\lim_{n\to\infty}(\|u_n^+\|_{L^\infty(B_r(x_0))}+\|\psi_n\|_{L^\infty(B_r(x_0))})=+\infty.$$ If not, then there exists $r_0>0$ such that $$\sup_n(\|u_n^+\|_{L^\infty(B_{r_0}(x_0))}+\|\psi_n\|_{L^\infty(B_{r_0}(x_0))})\leq C.$$ Now, we can choose $r_1<r_0$ such that $\sup_nE(u_n,\psi_n;B_{r_1}(x_0))<\epsilon_1$ which is a contradiction to $x_0\in S$ and the claim follows. Since $S$ is finite, we may choose $r$ small such that $B_{2r}(x_0)\cap \Sigma=\{x_0\}$.  Let $x_n\in B_r(x_0)$ be such that $$u_n^+(x_n)+|\psi_n(x_n)|=\sup_{x\in B_r(x_0)}(u_n^+(x)+|\psi_n|(x)).$$ We claim that $x_n\to x_0$. If not, then there exists a point $\overline{x}\neq x_0$ such that $x_n\to \overline{x}\in B_{2r}(x_0)$, which means $\overline{x}\in \Sigma$. Since $\Sigma\subset S$, then $\overline{x}\in S$ which implies $B_{2r}(x_0)\cap \Sigma=\{x_0,\overline{x}\}$. This is a contradiction.

\

\noindent\textbf{Step 2.} Conclusion $(2)$.

\

By Lemma \ref{lem:small-energy-regul}, it is easy to see that the case $S=\emptyset$ implies the conclusions $(a)$, $(b)$ and the case $S\neq \emptyset$ implies the conclusion $(c)$.
\end{proof}

\

Next we prove the following Pohozaev type identity which is very useful in our later analysis.

\begin{lem}\label{Pohozaev for Liouville}
  Let $(u,\psi)$ be a smooth solution to \eqref{equat:01},then for each $B_r \subset M$ we have:
   \begin{align*}
      R\int _{\partial B_R} \left(\left|
      \frac{\partial u}{\partial \nu}
    \right|^2 - \frac{1}{2}|\nabla u|^2\right)d\theta
    &= \int _{B_R} \left(2e^{2u} - e^u|\psi|^2 -|\psi|^4 x\cdot \nabla F\right)dx
    -R\int_{\partial B_R}\left( e^{2u} + F|\psi|^4\right)\\
     &\quad +\frac{1}{2}\int_{\partial B_R} \left(\left<
      x\cdot \psi,\frac{\partial \psi}{\partial \nu}
    \right>+\left<
      \frac{\partial \psi}{\partial \nu},x\cdot \psi
    \right>\right)-\int_{B_R} K_g x\cdot \nabla udx.
    \end{align*}
\end{lem}
\begin{proof}
  We choose a local orthonormal basis $e_1,e_2$ on $M$ such that $\nabla _{e_\beta} e_{\alpha} = 0$ at a considered point.
  Denote that $x = x_1 e_1 + x_2 e_2$. Multiplying the first equation by $x\cdot \nabla u$ and integrating over $B_R$, we have
  \begin{equation*}
    -\int _{B_R}\Delta u x\cdot\nabla u dx= \int_{B_R}2e^{2u}x\cdot\nabla u dx- \int_{B_R} e^u |\psi|^2 x\cdot\nabla u   dx -\int_{B_R} K_g x\cdot\nabla u dx.
  \end{equation*}
  Integrating by parts, we get
  \begin{align}\label{multiply x nabla u}
    R\int_{\partial B_R} \left(\left|
      \frac{\partial u}{\partial \nu}
    \right|^2- \frac{1}{2}|\nabla u|^2\right)&= -R\int_{\partial B_R}e^{2u} + \int_{B_R} 2e^{2u}dx +
    R\int_{\partial B_R}e^{u}|\psi|^2 \notag\\
    &\quad - 2\int_{B_R} e^{u}|\psi|^2dx - \int_{B_R}e^u x\cdot \nabla(|\psi|^2)dx
    -\int_{B_R} K_g x\cdot \nabla u dx.
  \end{align}

  Using the Schrodinger-Lichnerowicz formula $\slashed{D}^2 = -\Delta + \frac{1}{2}K_g$, we have
  \begin{equation*}
    \Delta \psi = \nabla_{e_\alpha}(e^u + 2F |\psi|^2)e_{\alpha}\cdot \psi - (e^u + 2F |\psi|^2)^2\psi + \frac{1}{2}K_g \psi.
  \end{equation*}

Since $$\langle \psi, e_\alpha\cdot\psi \rangle+\langle e_\alpha\cdot\psi, \psi \rangle=0,\ \ \alpha=1,2,$$ we get
  \begin{align}\label{two laplace of psi_1}
     &\int_{B_R} \left\langle
      \Delta\psi,x\cdot\psi
      \right\rangle
      +
      \int_{B_R} \left\langle
      x\cdot\psi,\Delta\psi
      \right\rangle dx\notag\\
      &=\int_{B_R}\sum_{\alpha,\beta}
      \left(
      \nabla_{e_{\alpha}}(e^u + 2F |\psi|^2)
      \right)
      \left(
      \left\langle
        e_{\alpha}\cdot \psi,e_{\beta}\cdot \psi
      \right\rangle
      +
      \left\langle
        e_{\beta}\cdot \psi,e_{\alpha}\cdot \psi
      \right\rangle
      \right)x_{\beta}dx\notag\\
      &=\int_{B_R} 2\sum_{\alpha}
      \left(
        \nabla_{e_{\alpha}}(e^u + 2F |\psi|^2)
      \right)|\psi|^2 x_{\alpha}dx\notag\\
      &=2\int_{B_R}\left(
        x \cdot \nabla (e^u + 2F |\psi|^2)
      \right)|\psi|^2dx\notag\\
      &=-2\int_{B_R}(e^u + 2F|\psi|^2)(x\cdot \nabla|\psi|^2)dx
      +2R\int_{\partial B_R}(e^u + 2F|\psi|^2)|\psi|^2
      - 4\int_{B_R}(e^u + 2F|\psi|^2)|\psi|^2dx.
  \end{align}

Integrating by parts, we obtain
     \begin{align*}
       \int_{B_R} \left\langle
      \Delta\psi,x\cdot\psi
      \right\rangle dx &=\int_{B_R}div \left\langle
        \nabla \psi,x\cdot\psi
      \right\rangle dx
      - \int_{B_R}\sum_{\alpha}\left\langle
        \nabla_{e_\alpha}\psi,e_\alpha \cdot\psi
      \right\rangle dx
      - \int_{B_R} \left\langle
        \nabla \psi,x\cdot \nabla\psi
      \right \rangle dx\\
      &=\int_{\partial B_R}\left\langle
        \frac{\partial \psi}{\partial \nu},x\cdot \psi
      \right\rangle
      + \int_{B_R}\left\langle
        \slashed{D}\psi,\psi
      \right\rangle dx
      - \int_{B_R}\left\langle
        \nabla \psi,x\cdot \nabla\psi
      \right\rangle dx\\
      &=\int_{\partial B_R}\left\langle
        \frac{\partial \psi}{\partial \nu},x\cdot \psi
      \right\rangle
      - \int_{B_R} (e^u + 2F|\psi|^2)|\psi|^2 dx
      - \int_{B_R}\left\langle
        \nabla \psi,x\cdot \nabla\psi
      \right\rangle dx,
     \end{align*}
     and
     \begin{equation*}
      \int_{B_R} \left\langle
      x\cdot\psi,\Delta\psi
      \right\rangle dx=\int_{\partial B_R}\left\langle
        x\cdot \psi,\frac{\partial \psi}{\partial \nu}
      \right\rangle
      - \int_{B_R} (e^u + 2F|\psi|^2)|\psi|^2 dx
      - \int_{B_R}\left\langle
        x\cdot \nabla\psi,\nabla \psi
      \right\rangle dx.
     \end{equation*}
     Adding above two equalities, there holds
     \begin{equation}\label{two laplace of psi_2}
      \int_{B_R} \left(\left\langle
      \Delta\psi,x\cdot\psi
      \right\rangle +
    \left\langle
      x\cdot\psi,\Delta\psi
      \right\rangle \right)dx
      =\int_{\partial B_R}\left\langle
        \frac{\partial \psi}{\partial \nu},x\cdot \psi
      \right\rangle
      +\left\langle
        x\cdot \psi,\frac{\partial \psi}{\partial \nu}
      \right\rangle
    -2\int_{B_R} (e^u + 2F|\psi|^2)|\psi|^2dx.
     \end{equation}
     From \eqref{two laplace of psi_1} and \eqref{two laplace of psi_2}, we have
       \begin{align}\label{equat:16}
      &-\int_{B_R} (e^u + 2F|\psi|^2)(x\cdot \nabla |\psi|^2) dx\notag\\
      &=\int_{B_R} (e^u + 2F|\psi|^2)|\psi|^2dx - R\int_{\partial B_R}(e^u + 2F|\psi|^2)|\psi|^2 + \frac{1}{2} \int_{\partial B_R}\left\langle
        \frac{\partial \psi}{\partial \nu},x\cdot \psi
      \right\rangle
      +\left\langle
        x\cdot \psi,\frac{\partial \psi}{\partial \nu}
      \right\rangle.
 \end{align}

Noting that
  \begin{align*}
    \int_{B_R} 2F |\psi|^2 x\cdot \nabla |\psi|^2 dx
    &=\int_{B_R} x\cdot \nabla (F|\psi|^4)dx
    - \int_{B_R} |\psi|^4 x\cdot \nabla F dx\\
    &= R\int_{\partial B_R} F|\psi|^4 - 2\int_{B_R} F|\psi|^4dx
    - \int_{B_R} |\psi|^4 x\cdot \nabla Fdx,
  \end{align*}
     then taking \eqref{equat:16} into \eqref{multiply x nabla u}, we obtain the conclusion of the lemma.
\end{proof}

\

At the end of this section, we give two lemmas about the estimates of spinor.

\begin{lem}\label{small_grad}
If $(v,\phi)$ is a smooth solution to
\begin{equation}\label{equat:06}
  \begin{cases}
    -\Delta v = 2e^{2v} - e^v |\phi|^2 + H_1(x) |\phi|^4,\\
    \slashed{D}\phi = -e^v\phi +H_2(x)|\phi|^2\phi,
\end{cases} \ \ in\ \  B_1(0)\setminus \{0\}
\end{equation}
  with
  \begin{equation*}
    E(v,\phi;B_1(0))\leq C,\ \ \sum_{i=1}^2\|H_i\|_{L^\infty(B_1(0))}\leq C,
  \end{equation*}
  then we have $    \norm{\nabla \phi}_{L^{4/3}(B_{1/2}(0))}\le C\norm{\phi}_{L^4(B_1(0))}.$
\end{lem}

\begin{proof}
  We choose a cut-off function $\eta_\epsilon\in C_0^\infty (B_{2\epsilon}(0))$ such that
    $\eta_\epsilon \equiv 1$ in $B_\epsilon(0)$ and $|\nabla \eta_\epsilon|<\frac{C}{\epsilon}$.
Then
\begin{equation*}
  \slashed{D}(1-\eta_\epsilon)\phi = (1-\eta_\epsilon)(-e^v\phi + H_2|\phi|^2\phi)- \nabla\eta_\epsilon \cdot \phi.
\end{equation*}
From the interior elliptic estimate of Dirac operator:
\begin{equation*}
  \begin{aligned}
  \norm{(1-\eta_\epsilon)\phi}_{W^{1,4/3}(B_{1/2})}&\le
  C \norm{(1-\eta_\epsilon)(e^v + H_2|\phi|^2) \phi}_{L^{4/3}(B_1)}
  +C\norm{\nabla \eta_\epsilon \cdot \phi}_{L^{4/3}(B_1)}
  +C\norm{(1-\eta_\epsilon)\phi}_{L^{4/3}(B_1)}\\
  &\le C\norm{(e^v + |\phi|^2)}_{L^2(B_1)}\norm{\phi}_{L^4(B_1)}
  +C\norm{\nabla \eta_\epsilon \cdot \phi}_{L^{4/3}(B_1)}
  +C\norm{\phi}_{L^{4/3}(B_1)}.
  \end{aligned}
\end{equation*}
Noting that
\begin{equation*}
  \begin{aligned}
  \norm{\nabla\eta_\epsilon \cdot \phi}_{L^{4/3}(B_1)}\le \frac{C}{\epsilon}\left(
    \int _{B_{2\epsilon}(0)\setminus B_{\epsilon}(0)} |\phi|^{4/3}dx
  \right)^{3/4}\le C\norm{\phi}_{L^4(B_1)},
  \end{aligned}
\end{equation*}
thus
\begin{equation*}
  \begin{aligned}
  \norm{\nabla \phi}_{L^{4/3}(B_{1/2})}&\le \lim_{\epsilon\rightarrow 0} \norm{(1-\eta_\epsilon)\phi}_{W^{1,4/3}(B_{1/2})}\\
  &\le C\norm{\phi}_{L^4(B_1)} + C\norm{(e^v + |\phi|^2)}_{L^2(B_1)}\norm{\phi}_{L^4(B_1)}\le C\norm{\phi}_{L^4(B_1)}.
  \end{aligned}
\end{equation*}
\end{proof}

\begin{lem}
 There exists a small constant $\epsilon_2>0$, such that if $(v,\phi)$ is a smooth solution to \eqref{equat:06} with
  \begin{equation*}
    E(v,\phi;B_2) = \int _{B_2} \left(e^{2v} + |\phi|^4\right)dx <\epsilon_2,\ \ \sum_{i=1}^2\|H_i\|_{L^\infty(B_2(0))}+\|\nabla H_2\|_{L^\infty(B_2(0))}\leq C
  \end{equation*} and  $$osc_{B_{\frac{1}{2}|y|}(y)}v\leq C,\ \ \forall\  \ y\in B_1(0),$$
  then for any $x\in B_{1/2}(0)$, we have
  \begin{equation}\label{estimate_phi_1}
    |\phi(x)||x|^{1/2} + |\nabla \phi(x)||x|^{3/2} \le C\left(
      \int _{B_{2|x|}(0)} |\phi(y)|^4dy
    \right)^{1/4}.
  \end{equation}
  Furthermore, if we assume that $v = -(1-\epsilon)\ln |x|+O(1)$ for some $\epsilon\in (0,1)$, then for any $x\in B_{1/2}$,
  we have
  \begin{equation}\label{estimate_phi_2}
    |\phi(x)||x|^{1/2} + |\nabla \phi(x)||x|^{3/2} \le C\left(
      \int _{B_1} e^{2v} + \int _{B_1} |\nabla \phi|^{4/3} + \int _{B_1} |\phi|^4
    \right)^{1/4}\cdot |x|^{1/4C}
  \end{equation}
  for some positive $C$.
\end{lem}
\begin{proof}
We follow the idea of \cite{Jost-Wang-Zhou}, but with some revisions to deal with the nonlinear spinorial Yamabe term.

  Fix any $x_0\in B_{1/2}\setminus \{0\}$ and define $(\tilde{v},\tilde{\phi})$ by
\begin{equation*}
  \begin{aligned}
    \tilde{v}(x) = v(x_0 + \frac{1}{2}|x_0|x)+ \log(\frac{1}{2}|x_0|),\ \
    \tilde{\phi}(x) = (\frac{1}{2}|x_0|)^{1/2} \phi(x_0 + \frac{1}{2}|x_0|x).
  \end{aligned}
\end{equation*}
 It is clear that $(\tilde{v},\tilde{\phi})$ is smooth solution to \eqref{equat:07}
 in $B_1$ with
 $$E(\tilde{v},\tilde{\phi},B_1)<\epsilon_2,\ \ osc_{\partial B_1(0)}\tilde{v}\leq C.$$ Taking $\epsilon_2<\epsilon_1$,  by Lemma \ref{lem:small-energy-regul}, we have $ \|\tilde{\phi}\|_{C^1(B_{1/2})} \le C\|\tilde{\phi}\|_{L^4(B_1)}, $ which immediately implies \eqref{estimate_phi_1}.

 Next we recall that $\phi$ satisfies
 \begin{equation*}
  \slashed{D}\phi = -(e^{2v} + 2H_2|\phi|^2)\phi\ \ in \ \ B_1\setminus \{0\}.
 \end{equation*}
 Similar to lemma \ref{small_grad}, we have
 \begin{equation*}
  \slashed{D}(1-\eta_\epsilon)\phi = -(1-\eta_\epsilon)(e^v + 2H_2 |\phi|^2)\phi - d\eta_\epsilon \cdot \phi.
\end{equation*}
From the elliptic estimate we have
\begin{equation*}
  \begin{aligned}
  \norm{(1-\eta_\epsilon)\phi }_{W^{1,4/3}(B_1)}&\le C\left(
    \norm{(1-\eta_\epsilon)(e^v +  |\phi|^2)\phi}_{L^{4/3}(B_1)}
    +\norm{d\eta_\epsilon \cdot \phi}_{L^{4/3}(B_1)}
    +\norm{(1-\eta_\epsilon)\phi }_{W^{1,4/3}(\partial B_1)}
  \right)\\
  &\le C \left(
    \norm{e^v +  |\phi|^2}_{L^2(B_1)}\norm{\phi}_{L^4(B_1)}
    +\norm{d\eta_\epsilon \cdot \phi}_{L^{4/3}(B_1)} +\norm{\phi}_{W^{1,4/3}(\partial B_1)}
  \right).
  \end{aligned}
\end{equation*}
From (\ref{estimate_phi_1}), we have
$
  \lim_{\epsilon \rightarrow 0}\frac{1}{\epsilon}\norm{\phi}_{L^{4/3}(B_{2\epsilon})}= 0.
$
Thus
  \begin{align}\label{inequ:04}
  \norm{\phi}_{L^4(B_1)}\le\lim_{\epsilon\rightarrow 0}\norm{(1-\eta_\epsilon)\phi }_{W^{1,4/3}(B_1)}&\le C\left(
    \norm{e^v +  |\phi|^2}_{L^2(B_1)}\norm{\phi}_{L^4(B_1)}
     +\norm{\phi}_{W^{1,4/3}(\partial B_1)}
  \right)\notag\\
  &\le C\left(
   \epsilon_2\norm{\phi}_{L^4(B_1)}
     +\norm{\phi}_{W^{1,4/3}(\partial B_1)}
  \right).
  \end{align}

Taking $\epsilon_2$ small, we have
\begin{equation*}
  \norm{\phi}_{L^4(B_1)} \le C \norm{\phi}_{W^{1,4/3}(\partial B_1)}.
\end{equation*}

By rescaling argument, we get
\begin{equation}\label{esti_phi_Br}
  \int _{B_r}|{\phi}|^4 \le Cr \left(
    \int _{\partial B_r} |\nabla \phi|^{4/3} + \int _{\partial B_r}|\phi|^4
  \right)\ \ \forall\ \ 0\le r \le 1.
\end{equation}
Denote $\bar{\phi} :=\frac{1}{|\partial B_1|}\int_{\partial B_1} \phi $,
then we have \begin{equation*}
  \slashed{D}(\phi - \bar{\phi}) = -(e^v + H_2 |\phi|^2)(\phi - \bar{\phi}) - (e^v+H_2 |\phi|^2)\bar{\phi}\ \,in \ \ B_1\setminus \{0\}.
\end{equation*}
Similar to deriving \eqref{inequ:04} and using Poincare inequality, we have
  \begin{align}\label{inequ:05}
  \norm{\phi - \bar{\phi}}_{W^{1,4/3}(B_1)}&\le C\left(
    \norm{(e^v + |\phi|^2)(\phi - \bar{\phi})}_{L^{4/3}(B_1)} +
    \norm{(e^v + |\phi|^2)\bar{\phi}}_{L^{4/3}(B_1)} +
    \norm{\phi - \bar{\phi}}_{W^{1,4/3}(\partial B_1)}
  \right)\notag\\
  &\le C\left(
    \norm{e^v + |\phi|^2}_{L^2(B_1)}\norm{\phi - \bar{\phi}}_{L^4(B_1)} +
    |\bar{\phi}|\cdot \norm{e^v + |\phi|^2}_{L^{4/3}(B_1)} + \norm{\nabla \phi}_{L^{4/3}(\partial B_1)}
  \right)\notag\\
  &\le C\epsilon_2 \norm{\phi - \bar{\phi}}_{W^{1,4/3}(B_1)}
  + C|\bar{\phi}| \cdot \norm{e^v + |\phi|^2}_{L^{4/3}(B_1)}
  + C\norm{\nabla \phi}_{L^{4/3}(\partial B_1)}.
  \end{align}

By Lemma  \ref{small_grad}, we have $$\|\nabla \phi\|_{L^{\frac{4}{3}}(B_1(0))}\leq C\|\phi\|_{L^4(B_2(0))}\leq C(\epsilon_2)^{\frac{1}{4}}.$$ Noticing that
\begin{equation*}
  \begin{aligned}
    |\bar{\phi}| \cdot \norm{e^v}_{L^{4/3}(B_1)}\le
    C\left(
      \int _{\partial B_1} |\phi|^4
    \right)^{1/4} \norm{e^v}_{L^2(B_1)}\le C\left(
      \int _{\partial B_1} |\phi|^4
    \right)^{3/4} + C\left(
      \int _{B_1} e^{2v}
    \right)^{3/4},
  \end{aligned}
\end{equation*}
and
\begin{equation*}
  \begin{aligned}
    |\bar{\phi}| \cdot \norm{|\phi|^2}_{L^{4/3}(B_1)}
    &\le C|\bar{\phi}| \cdot \norm{|\phi - \bar{\phi}|^2 + |\bar{\phi}|^2}_{L^{4/3}(B_1)}\\
    &\le C \norm{|\phi - \bar{\phi}|^2}^{\frac{3}{2}}_{L^{4/3}(B_1)} + C|\bar{\phi}|^3\\
    &\le C \norm{\phi - \bar{\phi}}^3_{L^{4}(B_1)}
    + C|\bar{\phi}|^3\\
    &\le C \norm{\nabla \phi}^2_{L^{4/3}(B_1)}\norm{\phi - \bar{\phi}}_{W^{1,4/3}(B_1)}
    + C\left(
      \int _{\partial B_1} |\phi|^4
    \right)^{3/4},
  \end{aligned}
\end{equation*}
putting these into \eqref{inequ:05} and taking $\epsilon_2$ small enough, we get
\begin{equation*}
  \norm{\phi - \bar{\phi}}_{W^{1,4/3}(B_1)}
  \le C\left(
    \int_{\partial B_1} |\phi|^4
  \right)^{3/4}
  + C\left(
    \int_{B_1} e^{2v}
  \right)^{3/4}
  +C\norm{\nabla  \phi}_{L^{4/3}(\partial B_1)}.
\end{equation*}
With a rescaling argument,  we have
\begin{equation*}
  \int_{B_r} |\nabla \phi|^{4/3}\le Cr \int _{\partial B_r}|\phi|^4 +C\int _{B_r} e^{2v} +Cr\int _{\partial B_r} |\nabla \phi|^{4/3}.
\end{equation*}
Noticing that $\frac{C_1}{|x|^{2-\epsilon}}\leq e^{2v} \leq \frac{C_2}{|x|^{2-\epsilon}}$, we have
$
  \int_{B_r} e^{2v} \le Cr\int_{\partial B_r} e^{2v}.
$
Thus \begin{equation*}
  \int _{B_r}e^{2v}+\int _{B_r} |\nabla \phi|^{4/3} + \int _{B_r} |\phi|^4
  \le Cr\left(
    \int_{\partial B_r} e^{2v} +|\nabla \phi|^{4/3} + |\phi|^4
  \right).
\end{equation*}
Denote
\begin{equation*}
  F(r):=\int _{B_r}e^{2v}+\int _{B_r} |\nabla \phi|^{4/3} + \int _{B_r} |\phi|^4,\ \ \forall 0\le r \le 1,
\end{equation*}
then $  F(r)\le Cr\cdot F'(r)$, which yields
\begin{equation*}
  F(r)\le F(1)\cdot r^{1/c}.
\end{equation*} Then \eqref{estimate_phi_2} follows from \eqref{estimate_phi_1}.
\end{proof}

\

\section{Removability of local singularities}\label{sec:removability of local singularities}

\

In this section, we shall study the local singularity for super-Liouville equation with a spinorial Yamabe type term. On one hand, we will establish some asymptotic estimates near the singularity. On the other hand, similar to \cite{Jost-Zhou-Zhu}, by defining a so-called Pohozaev type constant,  we  show that a local singularity is removable if and only if the Pohozaev identity is satisfied.

We consider the following super-Liouville equation with a spinorial Yamabe type term in $B_1(0)\setminus \{0\} $ that
\begin{equation}\label{Liouville on Br-0}
    \begin{cases}
    -\Delta u = 2e^{2u} - e^u |\psi|^2 ,\\
    \slashed{D}\psi = -e^u\psi - 2F |\psi|^2\psi,
\end{cases}\ \  x \in   B_1(0)\setminus \{0\}.
  \end{equation}

\

We first derive the following Pohozaev type constant.
\begin{lem}\label{lem:pohozaev-const}
  Let $(u,\psi)\in C^2(B_1(0)\setminus \{0\}) \times C^2(\Gamma(\Sigma(B_1(0)\setminus \{0\})))$ be solution of (\ref{Liouville on Br-0}) with $    E(u,\psi; B_1)<\infty$. Then  for $0<R<1$, the following quantity
    \begin{align*}
    C(u,\psi;R)
    := &R\int _{\partial B_R} \left(\left|
      \frac{\partial u}{\partial \nu}
    \right|^2 - \frac{1}{2}|\nabla u|^2\right)- \int _{B_R} \left(2e^{2u} - e^u|\psi|^2 -|\psi|^4 x\cdot \nabla F\right)dx
    \\
    &+R\int_{\partial B_R} \left(e^{2u} + F|\psi|^4\right)
     -\frac{1}{2}\int_{\partial B_R} \left(\left<
      x\cdot \psi,\frac{\partial \psi}{\partial \nu}
    \right>+\left<
      \frac{\partial \psi}{\partial \nu},x\cdot \psi
    \right>\right)
    \end{align*}
is independent of $R$. We call $C(u,\psi)=C(u,\psi;R)$ a Pohozaev type constant for $(u,\psi)$ near the singular point $0$.
\end{lem}
\begin{proof}
It is clear that for all $0<t\le r < 1$, $(u,\psi)\in C^2(B_r(0)\setminus B_t(0)) \times C^2(\Gamma(\Sigma(B_r(0)\setminus  B_t(0))))$.
We can integrate in the domain $B_r(0)\setminus B_t(0)$ and make the same progress as in the proof of Pohozaev identity.
\end{proof}

\

Next we show that the local singularity is removable if and only if the Pohozaev type constant vanishes.

  \begin{prop}(Removability of local singularity)\label{Removability of local singularity}
    Let $(u,\psi)\in C^2(B_1(0)\setminus \{0\}) \times C^2(\Gamma(\Sigma(B_1(0)\setminus \{0\})))$ be a solution of
    (\ref{Liouville on Br-0}) with $    E(u,\psi; B_1(0))<\infty$. Then the following alternatives hold:
    \begin{itemize}

    \item[(1)] There is a constant $\gamma<2\pi$ such that
    \begin{equation*}
      u(x) = -\frac{\gamma}{2\pi} \log|x| + h,\ \ near\ \  0
    \end{equation*}
    where $h$ is Holder continues near 0.

    \

    \item[(2)] There exists a small positive constant $\epsilon>0$ such that $$\nabla u=-\frac{\gamma}{2\pi}\frac{x}{|x|^2}+O(|x|^{-1+\epsilon}),\ \ |\psi|\leq C|x|^{-\frac{1}{2}+\epsilon},\ \ |\nabla\psi|\leq C|x|^{-\frac{3}{2}+\epsilon} \ \ near\ \ 0.$$

    \

    \item[(3)] The Pohozaev constant $C(u,\psi)$ and $\gamma$ satisfy
    \begin{equation*}
      C(u,\psi) = \frac{\gamma^2}{4\pi}.
    \end{equation*}
    In particular, the local singularity is removable iff $C(u,\psi)=0$.
    \end{itemize}
  \end{prop}
  \begin{proof}
     Without loss of generality, we assume that
     \begin{equation}\label{inequ:02}
     \int_{B_1(0)} \left(e^{2u} + |\psi|^4\right) dx\le \epsilon_0,\end{equation} where $\epsilon_0$ is a small constant which will be determined later. Next, we divide the proof into three steps.

     \

    \textbf{Step 1:} We prove that:
    \begin{equation}\label{inequ:01}
      u(x) \le -\log|x| + C,\ \         |\psi(x)| \le C |x|^{-1/2},\forall x \in B_{1/2}(0)\setminus \{0\}.
    \end{equation}
    For any $x_0\in B_{1/2}(0)\setminus \{0\} $, we have
    \begin{equation*}
      \int_{B_{\frac{|x_0|}{2}} (x_0)} \left(e^{2u}+|\psi|^4\right)dx \le \int_{B_1(0)}\left(e^{2u}+|\psi|^4\right)dx \le \epsilon_0.
    \end{equation*}
    Let
    \begin{equation*}
        \tilde{u}(x) = u(x_0 + \frac{1}{2}|x_0| x) + \ln\frac{|x_0|}{2},\ \
        \tilde{\psi}(x) = \left(
          \frac{|x_0|}{2}
        \right)^{1/2} \psi(x_0 + \frac{1}{2}|x_0|x).
    \end{equation*}
    Then we have
    \begin{equation*}
      \int_{B_1(0)} \left(e^{2\tilde{u}}+|\tilde{\psi}|^4\right)dx  \le \epsilon_0
    \end{equation*}
    and
    \begin{equation*}
      \begin{cases}
      -\Delta \tilde{u} = 2e^{2\tilde{u}} - e^{\tilde{u}}|\tilde{\psi}|^2 ,\\
        \slashed{D} \tilde{\psi} = -(e^{\tilde{u}} + 2G |\tilde{\psi}|^2) \tilde{\psi},
      \end{cases} \ \  x\in B_1,
    \end{equation*} where  $G(x)=F(x_0 + \frac{1}{2}|x_0| x)$.

    By Lemma \ref{lem:small-energy-regul}, we get
    $$\tilde{u}(x)\le C \ \ and\ \  |\tilde{\psi}(x)| \le C,\ \  \forall\  x\in B_{1/2}(0).$$
    Specially, there holds $\tilde{u}(0)\le C$ and $\tilde{\psi}(0)\le C$.
    This is $$u(x_0) \le -\log|x_0| +C\ \  and\ \  \psi(x_0) \le C|x_0|^{-1/2}.$$

    \

    \textbf{Step 2:} We prove that: there exists a constant $\gamma<2\pi$ such that
    \begin{equation}\label{equat:04}
    \lim_{x\rightarrow 0} \frac{u(x)}{-\log|x|} = \frac{\gamma}{2\pi}.
    \end{equation}

    \

    Let $f:=2e^{2u} - e^u |\psi|^2 $ and $v(x)$ be the solution of
    \begin{align*}
    \begin{cases}
    -\Delta v(x)=f(x)\ \ &in \ \ B_1(0),\\
    v(x)=0,\ \ &on\ \ \partial B_1(0).
    \end{cases}
    \end{align*}
    Then
    \begin{equation}\label{equat:05}
      v(x)=-\frac{1}{2\pi}\int_{B_1(0)} \log |x-y| f(y)dy+h(x),
    \end{equation} where $h(x)\in C^1(B_1(0))$.

By a standard potential analysis, we have
    \begin{equation}\label{equat:03}
      \lim_{x\rightarrow 0} \frac{v(x)}{\ln|x|} = 0.
    \end{equation}

    Since $u-v$ is a harmonic function in $B_1(0)\setminus \{0\}$ and $u-v\leq -C\ln |x|+C$, we know that there exists a constant $\gamma$ such that
    \begin{equation}\label{equat:08}
      u -v = -\frac{\gamma}{2\pi} \log|x| + w_0
    \end{equation}
    where $w_0$ is a smooth harmonic in $B_1(0)$. This immediately implies \eqref{equat:04}. Since $e^{2u}\in L^1(B_1(0))$, it is clear that $\frac{\gamma}{2\pi}<1$.

\

    \textbf{Step 3:} We prove that: there exists a small positive constant $\epsilon>0$ such that
    \begin{equation}
     u(x) = -\frac{\gamma}{2\pi} \log|x| + h,\ \ \nabla u= -\frac{\gamma}{2\pi} \frac{x}{|x|^2}+O(|x|^{-1+\epsilon}).
    \end{equation}

    \

By \eqref{inequ:02} and  Theorem 1 in \cite{Brezis}, if we take $\epsilon_0$ small, then we have $e^{|v|}\in L^p(B_1(0))$ for some $p> 100$. Using  \eqref{equat:08}, we get
\begin{align*}
e^{2u}+e^{u}|\psi|^2\leq C|x|^{-\frac{\gamma}{\pi}}e^{2v}+C|x|^{-\frac{\gamma}{2\pi}-1}e^{v},
\end{align*}  which implies $$e^{2u}+e^{u}|\psi|^2\in L^q(B_1(0)),\ \ for\ \ some\  \ q>1.$$

Then $f\in L^q(B_1(0))$ and the standard elliptic estimate yields $v\in C^\beta(B_1(0))$ for some $\beta\in (0,1)$. By  \eqref{equat:08}, we get the first conclusion of the proposition and following estimate $$osc_{B_{\frac{1}{2}|y|}(y)}u\leq osc_{B_{\frac{1}{2}|y|}(y)} \left(-\frac{\gamma}{2\pi}\ln |x|\right)+C\leq C,\ \ \forall \ y\in B_{\frac{1}{2}}(0).$$

  It is follows from estimate (\ref{estimate_phi_1}) and (\ref{estimate_phi_2}) that there exists a constant $\delta >0$ such that
  \begin{equation*}
    \begin{aligned}
    |\psi| &\le C|x|^{-\frac{1}{2}+\delta},\ \
    |\nabla \psi| &\le C  |x|^{-\frac{3}{2}+\delta}.
    \end{aligned}
  \end{equation*}

Next we claim: there exists a small positive constant $\epsilon>0$ such that
\begin{equation}\label{inequ:03}
|\nabla v(x)|\leq C|x|^{-1+\epsilon}.
\end{equation}

Since now we have $|f(x)|\leq C|x|^{-s}$ for some $0<s<2$ and $f\in L^q(B_1(0))$, then we get
  \begin{equation*}
    \begin{aligned}
    \int_{B_1(0)} \frac{1}{|x-y|}|f(y)|dy
    =\int_{\{|x-y|\ge \frac{|x|}{2}\}} \frac{1}{|x-y|}|f(y)|dy
    + \int_{\{|x-y|\le \frac{|x|}{2}\}} \frac{1}{|x-y|}|f(y)|dy=\mathbf{II_1} + \mathbf{II_2}.
    \end{aligned}
  \end{equation*}
  Denoting $q'=\frac{q}{q-1}$, then we have
  \begin{equation*}
    \mathbf{II_1} =\int_{\{|x-y|\ge \frac{|x|}{2}\}} |x-y|^{-1+\frac{1}{q'}}|x-y|^{-\frac{1}{q'}}|f(y)|dy\leq C|x|^{-1+\frac{1}{q'}}\int_{B_1(0)} |x-y|^{-\frac{1}{q'}}|f(y)|dy\leq C|x|^{-1+\frac{1}{q'}}
  \end{equation*}
and
  \begin{equation*}
    \mathbf{II_2}=\int_{\{|x-y|\le \frac{|x|}{2}\}} \frac{1}{|x-y|}|f(y)|dy \le \int_{\{|x-y|\le \frac{|x|}{2}\}} \frac{1}{|x-y|}\cdot |x|^{-s}dy\le C|x|^{1-s}.
  \end{equation*}
Combining above estimates and \eqref{equat:05}, we get  \eqref{inequ:03}.

Thus,
  \begin{equation*}
    \nabla u = -\frac{\gamma}{2\pi} \frac{x}{|x|^2} + \nabla (w_0 + v(x)) = -\frac{\gamma}{2\pi} \frac{x}{|x|^2}+O(|x|^{-1+\epsilon}).
  \end{equation*}

With the help of above estimates, we have that
    \begin{align*}
    R\int_{\partial B_R}
    \left|
      \frac{\partial u}{\partial \nu}
    \right|^2 - \frac{1}{2}|\nabla u|^2
    =\frac{\gamma^2}{4\pi} + o_R(1),\ \
    &\int_{B_R} \left(2e^{2u} -e^u|\psi|^2 -|\psi|^4 x\cdot \nabla F\right)dx   = o_R(1),\\
    R\int_{\partial B_R} e^{2u}+F|\psi|^4 = o_R(1),\ \ \ \ \ \ \ \
    &\frac{1}{2}\int_{\partial B_R} \left<
      x\cdot \psi,\frac{\partial \psi}{\partial \nu}
    \right>+\left<
      \frac{\partial \psi}{\partial \nu},x\cdot \psi
    \right> =o_R(1),
    \end{align*} where $\lim_{R\to 0}o_R(1)=0$.

By Lemma \ref{lem:pohozaev-const}, we get
\begin{equation*}
    C(u,\psi)=\lim_{R\rightarrow 0} C(u,\psi;R)=\frac{\gamma^2}{4\pi},
  \end{equation*} which implies the third conclusion of proposition.
\end{proof}

\

\section{Bubble's equation}\label{sec:Bubble's equation}

\

In this section, we will study the properties of the equation of super-Liouville type bubble which is defined on the whole plane. Specially, we will show that the energy of this type bubble is $4\pi$ and the singularity at infinity is removable.

The bubble equation is
\begin{equation}\label{bubble-eq}
  \begin{cases}
      -\Delta u &= 2e^{2u} - e^u |\psi|^2 ,\\
      \slashed{D}\psi &= -e^u\psi - 2\mu|\psi|^2\psi,
  \end{cases}in\ \  \mathbb{R}^2,
\end{equation}
where $\mu\in \mathbb{R}$
with energy condition $  E(u,\psi;\mathbb{R}^2)<\infty. $

\begin{lem}\label{lem:03}
  Let $(u,\psi)$ be a smooth solution to (\ref{bubble-eq}) with  finite energy
  $E(u,\psi;\mathbb{R}^2) <\infty$. Then $u^+ \in L^{\infty} (\mathbb{R}^2)$.
\end{lem}

\begin{proof}
 Since $E(u,\psi; \mathbb{R}^2) <\infty$, for any $\epsilon>0$, there exists $R=R(\epsilon)>0$ such that for any $|y|\geq R$, there holds $E(u,\psi ;B_1(y))\leq \epsilon$. By the proof of Lemma \ref{lem:small-energy-regul}, we get that $$\max_{B_{\frac{1}{2}}(y)}u\leq C(1+\|u^+\|_{L^1(B_1(y))})\leq C(1+E(u,\psi ;\mathbb{R}^2)),$$ where $C$ is independent of $y$, which implies that  $u^+ \in L^{\infty} (\mathbb{R}^2)$.
\end{proof}

\

Next, we will combine the potential analysis with Pohozaev type identity to calculate the bubble's energy $$\alpha := \int_{\mathbb{R}^2} \left(2e^{2u} - e^u |\psi|^2 \right)dx.$$

\begin{prop}\label{prop:01}
  Let $(u,\psi)$ be a solution to (\ref{bubble-eq}) with finite energy $E(u,\psi;\mathbb{R}^2) <\infty$. Then we have
  \begin{align*}
    u(x) = -\frac{\alpha}{2\pi} \log |x| + O(1),\ \  \nabla u(x)=-\frac{\alpha}{2\pi}\frac{x}{|x|^2} +O(|x|^{-1-\epsilon})\ \  near\ \  \infty,
    \end{align*} where $\alpha=4\pi$.

\end{prop}
\begin{proof}

  Let $f:= 2e^{2u} - e^u |\psi|^2$ and
  \begin{equation*}
    v(x) := -\frac{1}{2\pi} \int_{\mathbb{R}^2} \left(
      \log |x-y| - \log(|y|+1)
    \right)f(y)dy.
  \end{equation*}
It is clear that $-\Delta v = f$
  and \begin{equation*}
    \lim_{|x|\rightarrow \infty} \frac{v}{\log |x|} = -\frac{\alpha}{2\pi}.
  \end{equation*}
  Considering $w = u-v$ we have $\Delta w = 0$.
Lemma \ref{lem:03} tells us that  $u$ is bounded from above. Then we have
  \begin{equation*}
    w\le  C+C\log(1+|x|).
  \end{equation*}
  By Liouville theorem, $w$ is a constant on $\mathbb{R}^2$.
  Then \begin{equation}\label{equat:21}
    \lim_{x\rightarrow \infty} \frac{u(x)}{\log |x|} = \lim_{x\rightarrow \infty} \frac{v(x)}{\log |x|} = -\frac{\alpha}{2\pi}.
   \end{equation}

   Considering the Kelvin transform $(u_1,\psi_1)$ of $(u,\psi)$, i.e. $$u_1(x):=u(\frac{x}{|x|^2})-2\ln |x|,\ \ \psi_1(x):=|x|^{-1}\psi(\frac{x}{|x|^2}),$$ we have
   \begin{equation*}
    \begin{cases}
      -\Delta u_1 &= 2e^{2u_1} - e^{u_1} |\psi_1|^2 ,\\
      \slashed{D}\psi_1 &= -e^{u_1}\psi_1 - 2\mu|\psi_1|^2\psi_1,
    \end{cases}\ \ x \in B_1 \setminus \{0\}.
   \end{equation*}

    By Proposition \ref{Removability of local singularity},
    we have there exist two constants $\gamma<2\pi $ and $\epsilon\in (0,\frac{1}{2})$ such that
    \begin{equation*}
      u_1(x) = -\frac{\gamma}{2\pi} \log |x| + h(x),\ \ \nabla u_1(x)=-\frac{\gamma}{2\pi}\frac{x}{|x|^2}+O(|x|^{-1+\epsilon})\ \ near\ \ 0,
    \end{equation*} and
       \begin{equation}\label{equat:22}
      |\psi_1(x)| \le C |x|^{-\frac{1}{2}+\delta},\ \
        |\nabla \psi_1(x)| \le C |x|^{-\frac{3}{2}+\delta }\ \ near\ \ 0,
    \end{equation} where $h(x)$ is continues near $0$.
    From the fact that
    \begin{equation*}
      u(x) = u_1(x/|x|^2) - 2\log |x|=\left(
        \frac{\gamma}{2\pi}-2
      \right)\log|x| + h(\frac{x}{|x|^2})=\left(
        \frac{\gamma}{2\pi}-2
      \right)\log|x| + O(1) \ \ near\  \ \infty,
    \end{equation*} by \eqref{equat:21},
    we get \begin{equation*}
      \alpha + \gamma = 4\pi.
    \end{equation*} Thus, we have $$ u(x) = -\frac{\alpha}{2\pi} \log |x| + O(1),\ \   near\ \  \infty,$$ where $\alpha=4\pi-\gamma>2\pi$.

    By a direct computation, we have \begin{equation*}
      \begin{aligned}
        (\nabla u) (x) = -(\nabla u_1) (\frac{x}{|x|^2}) \cdot \frac{1}{|x|^2} - \frac{2}{|x|^2} x\ \
        &=(\frac{\gamma}{2\pi}-2) \frac{x}{|x|^2} +O(|x|^{-1-\epsilon})\\&= -\frac{\alpha}{2\pi}\frac{x}{|x|^2} +O(|x|^{-1-\epsilon})\ \ near\ \ \infty.
      \end{aligned}
    \end{equation*}

 \

Next, we \textbf{claim} $\alpha=4\pi$.

 \

 From \eqref{equat:22}, we have $$|\psi(x)| \le C |x|^{-\frac{1}{2}-\delta},\ \
        |\nabla \psi(x)| \le C |x|^{-\frac{3}{2}-\delta }\ \ near\ \ \infty.$$
Using above estimates, we have
\begin{align*}
    R\int_{\partial B_R}\left(\left|
      \frac{\partial u}{\partial \nu}
    \right|^2 - \frac{1}{2}|\nabla u|^2\right)d\theta &= \frac{\alpha^2}{4\pi}+o(R),\\
    R\int_{\partial B_R}\left(e^{2u}+\mu |\psi|^4\right)d\theta &\le CR\int_{\partial B_R}\left(|x|^{-\alpha/\pi} + |x|^{-2-4\delta}\right)d\theta=o(R),\\
    \left|\int_{\partial B_R}\left\langle
      x\cdot\psi,\frac{\partial\psi}{\partial \nu}
    \right\rangle d\theta\right| &\le CR\int_{\partial B_R} |\psi||\nabla\psi|d\theta\le CR\int_{\partial B_R} |x|^{-\frac{1}{2}-\delta}|x|^{-\frac{3}{2}-\delta}d\theta=o(R),
  \end{align*} where $\lim_{R\to \infty}o(R)=0$.

By Pohozaev type identity for bubble's equation that
    \begin{align*}
      R\int _{\partial B_R} \left(\left|
      \frac{\partial u}{\partial \nu}
    \right|^2 - \frac{1}{2}|\nabla u|^2\right)d\theta
    &=\int _{B_R} \left(2e^{2u} - e^u|\psi|^2 \right)dx
    - R\int_{\partial B_R}\left( e^{2u} + \mu |\psi|^4 \right)d\theta \\
    &
    \quad+ \frac{1}{2}\int_{\partial B_R} \left(\left<
      x\cdot \psi,\frac{\partial \psi}{\partial \nu}
    \right>+\left<
      \frac{\partial \psi}{\partial \nu},x\cdot \psi
    \right>\right)d\theta,
    \end{align*}
letting $R\to \infty$, we get
  \begin{equation*}
    \frac{\alpha^2}{4\pi} = \alpha,
  \end{equation*} which implies $\alpha=4\pi$.
  \end{proof}

\

By above proposition, we have the following statement for the removability of singularity at infinity.
\begin{cor}\label{removability at infinity}
  Let $(u,\psi)$ be a smooth solution to \eqref{bubble-eq} with finite energy $E(u,\psi;\mathbb{R}^2)<\infty$.
  Then $(u,\psi)$ can be extend to a smooth solution of \eqref{equat:10} on $\mathbb{S}^2$, i.e. the singularity at infinity is removable.
\end{cor}
\begin{proof}
  Considering the stereographic projection
\begin{align*}
  f:(\mathbb{S}^2\setminus \{N\},g_{\mathbb{S}^2} = d\theta^2 + \sin^2 \theta d\varphi^2 )&\rightarrow (\mathbb{R}^2,g_{\mathbb{R}^2} = dx^2 + dy^2)\\
  (\theta,\varphi )&\mapsto \left(
    \frac{\sin \theta \cos \varphi}{1-\cos \theta}, \frac{\sin \theta \sin \varphi}{1-\cos \theta}
  \right)
\end{align*}
It is clear that
\begin{equation*}
  f^*(g_{\mathbb{R}^2}) = \left(d
    \frac{\sin \theta \cos \varphi}{1-\cos \theta}
  \right)^2 + \left(
    d\frac{\sin \theta \sin \varphi}{1-\cos \theta}
  \right)^2 = \frac{g_{\mathbb{S}^2}}{(1-\cos \theta )^2}:=\lambda^2 g_{\mathbb{S}^2}.
\end{equation*}
Define that
\begin{equation*}
    v(\theta,\varphi) = u\circ f(\theta,\varphi) + \log \lambda, \ \
    \phi(\theta,\varphi) = \lambda^{\frac{1}{2}}\psi\circ f(\theta,\varphi),
\end{equation*}
which satisfy the following equations
\begin{equation*}
  \begin{cases}
    -\Delta v &= 2e^{2v} - e^v |\phi|^2-1, \\
      \slashed{D}\phi &= -e^v\phi - 2\mu|\phi|^2\phi,
  \end{cases} \ \ on \ \ \mathbb{S}^2 \setminus \{N\}.
\end{equation*}
For  $u = -2\log |x| + O(1)$ as $x$ near infinity, we have
  \begin{align*}
    v(\theta,\varphi) &= u\left(
    \frac{\sin \theta \cos \varphi}{1-\cos \theta}, \frac{\sin \theta \sin \varphi}{1-\cos \theta}
  \right)+ \log \frac{1}{1-\cos \theta}\\
    &=-2\log \frac{|\sin \theta|}{1-\cos\theta}+ \log \frac{1}{1-\cos \theta} + O(1)\\
    &= \log\frac{1}{1+\cos\theta}+O(1)\\
    &=O(1), \ \ as \ \ \theta \rightarrow 0.
  \end{align*}
Thus $v$ is bounded near $\{N\}\in \mathbb{S}^2$.
From the standard elliptic establish we can claim that $(v,\phi)$ is regular on $\mathbb{S}^2$.
\end{proof}

\

  \section{Energy identity for the spinor}\label{sec:energy-identity-spinor}

  \

  In this section we will prove energy identity for the spinor, i.e. Theorem \ref{thm:main-2}, which is important in later proof, such as local masses and energy identity for the function.

  \

  We first show the following elliptic estimate for spinor's equation.
  \begin{lem}\label{lem:07}
    let $(u,\psi)$ be a smooth solution of \eqref{equat:01} on the annulus $A_{r_1,r_2}=\{x\in \mathbb{R}^2:r_1\le |x| \le r_2\}$,where $0<r_1 <2r_1 <r_2 /2 <r_2<1$.
    Then we have
      \begin{align}\label{esti-psi-annulus}
        &\left(
        \int_{A_{2r_1,\frac{r_2}{2}}} |\nabla \psi|^{\frac{4}{3}}dx
      \right)^{\frac{3}{4}}
      +
      \left(
        \int_{A_{2r_1,\frac{r_2}{2}}} |\psi|^4dx
      \right)^{\frac{1}{4}}\notag\\
      &\le \Lambda \left(
        \int_{A_{r_1,r_2}} (e^{2u} + |\psi|^4)dx
      \right)^{\frac{1}{2}}
      \left(
        \int_{A_{r_1,r_2}} |\psi|^4dx
      \right)^{\frac{1}{4}}
      +
      \left(
      C\int_{A_{r_1,2r_1}} |\psi|^4dx
      \right)^{\frac{1}{4}}
      +C\left(
      \int_{A_{\frac{r_2}{2},r_2}} |\psi|^4dx
      \right)^{\frac{1}{4}},
      \end{align}
  \end{lem} where $\Lambda>0,\ C>0$ are two universal constants.
\begin{proof}
The proof is standard by using elliptic estimates of Dirac operator. See Lemma 3.1 in \cite{Jost-Wang-Zhou-Zhu-1}.
\end{proof}

\

Now, we prove energy identity for the spinor, i.e. Theorem \ref{thm:main-2}. Since $\Sigma$ is finite, we  choose $\{\delta_i\}$  small such that
  $x_i\in D_{\delta_i}(x_i)$ and $D_{\delta_i}(x_i)\cap D_{\delta_j}(x_j) = \emptyset$ as $i\neq j$. By Lemma \ref{lem:small-energy-regul}, we know that $\psi_n$ converges to $\psi$ in $L_{loc}^\infty(M\setminus \Sigma)$.
  It is clear that the energy identity is equivalent to
  \begin{equation*}
    \lim_{\delta\rightarrow 0}\lim_{n\rightarrow\infty}\int_{D_{\delta}(x_i)} |\psi_n|^4 dx= \sum_{k=1}^{L_i} \int_{S^2} |\xi^{i,k}|^4dx
  \end{equation*}
  where $\xi^{i,k}$ are bubbles.

\

We start the proof by constructing  the first bubble at  a blow-up point $p\in\Sigma$. Suppose $p$ is the only blow-up point in $B_\delta(p)$. We have to consider following two cases according to different singularities.

\

\textbf{Case (i):} $p\in\Sigma$ is a first type singular point.

\

In this case, we define
  \begin{align*}
   u_n (x_n):=\max_{x\in B_\delta(p) }\{ u_n (x)\},\  \
    \lambda_n := e^{-u_n(x_n)}.
  \end{align*}
  it is clear that $x_n \rightarrow p$ and $\lambda_n\rightarrow 0$ as $n\rightarrow \infty$.

  Setting
  \begin{equation*}
      \tilde{u}_n(x) = u_n (\lambda_n x+ x_n) + \log\lambda_n,\ \
      \tilde{\psi}_n(x) = \lambda_n^{\frac{1}{2}} \psi_n (\lambda_n x + x_n),
  \end{equation*}
  then
  \begin{equation*}
    \begin{cases}
    -\Delta \tilde{u}_n (x) &= 2e^{2\tilde{u}_n(x)} - e^{\tilde{u}_n(x)} |\tilde{\psi}_n(x)|^2 ,\\
      \slashed{D}\tilde{\psi}_n(x) &= -e^{\tilde{u}_n(x)}\tilde{\psi}_n(x) - 2F( \lambda_n x + x_n) |\tilde{\psi}_n(x)|^2\tilde{\psi}_n(x),
    \end{cases} \ \ x\in B_{\frac{\delta}{2\lambda_n}}(0),
  \end{equation*}
  with finite energy condition
  \begin{equation*}
    \int_{B_\frac{\delta}{2\lambda_n}(0)}( e^{2\tilde{u}_n(x)} + |\tilde{\psi}_n|^4 )dx< C.
  \end{equation*}

  From the definition of the first type singularity, we know that passing to a subsequence, there exists a universal constant $C>0$ such that $$u_n(x_n)-2\ln (1+\max_{B_{r_0}(p)}|\psi_n(x)|)\geq -C,\ \ \forall\ n.$$ Then it is easy to see that for any $R>0$, there holds
  \begin{align*}
   \tilde{u}_n(0)=0,\  \tilde{u}_n(x)\leq 0,\ \     |\tilde{\psi}_n(x)|^2\leq e^{-u_n(x_n)+2\ln \max_{B_{r_0}(p)}|\psi_n(x)|}\leq C,\ \  x\in B_R(0).
  \end{align*}

  By standard elliptic estimate,
  $(\tilde{u}_n,\tilde{\psi}_n)$ sub-converges to $(\tilde{u},\tilde{\psi})$  in $C^2_{loc}(\mathbb{R}^2)$ where
  \begin{equation*}
    \begin{cases}
      -\Delta \tilde{u} (x) &= 2e^{2\tilde{u}(x)} - e^{\tilde{u}(x)} |\tilde{\psi}(x)|^2 ,\\
      \slashed{D}\tilde{\psi}(x) &= -e^{\tilde{u}(x)}\tilde{\psi}(x) - 2\mu |\tilde{\psi}(x)|^2\tilde{\psi}(x),
    \end{cases}\ \ x\in\R^2.
  \end{equation*}

From the Corollary \ref{removability at infinity}, we get the first bubble at  blow up point $p$ which is a super-Liouville type bubble.

\

\textbf{Case (ii):} $p\in\Sigma$ is a second type singular point.

\

In this case, we define
  \begin{align*}
   |\psi_n (x_n)|:=\max_{x\in B_\delta(p) }\{ |\psi_n (x)|\},\  \
    \lambda_n := |\psi_n(x_n)|^{-2}.
  \end{align*}
  it is clear that $x_n \rightarrow p$ and $\lambda_n\rightarrow 0$ as $n\rightarrow \infty$.

Set $(\tilde{u}_n,\tilde{\psi}_n)$ as in Case $(i)$. Since $p\in\Sigma$ is of second type singularity, passing to a subsequence, we have $$\max_{B_{r_0}(p)}u_n(x)-2\ln (1+|\psi_n(x_n)|)\to -\infty,\ \ as\ \ n\to\infty.$$ It is easy to see that $$ |\tilde{\psi}_n(0)|=1,\ \ |\tilde{\psi}_n(x)|\leq 1,\ \ x\in B_R(0)$$ and $$ \max_{x\in B_R(0)}\tilde{u}_n(x)\leq \max_{B_{r_0}(p)}u_n(x)-2\ln |\psi_n(x_n)|\to -\infty,\ \ as\  \ n\to\infty.$$
By standard elliptic estimate, passing to a subsequence, we know that $\tilde{u}_n\to -\infty$ uniformly in any compact subset of $\R^2$ and
  $\tilde{\psi}_n$ converges to $\tilde{\psi}$  in $C^2_{loc}(\mathbb{R}^2)$ where
  \begin{equation}\label{equat:20}
      \slashed{D}\tilde{\psi}(x) =  - 2\mu |\tilde{\psi}(x)|^2\tilde{\psi}(x),\ \ x\in\R^2.
  \end{equation}
  By the conformal invariance of \eqref{equat:20} and removable singularity theorem (see Theorem 5.1 in \cite{Jost-Liu-Zhu}), we get the first bubble at blow-up point $p$ which is a spinorial Yamabe type bubble.

  \

Similar to the argument for the energy identity of harmonic maps, e.g. \cite{DingWeiyueandTiangang}, we may assume that there is only one bubble at each blow up point $p\in\Sigma$.  Then the energy identity is equivalent to prove
  \begin{equation}\label{equat:23}
    \lim_{R\rightarrow 0}\lim_{\delta \rightarrow 0}\lim_{n\rightarrow \infty} \int_{B_\delta(x_n)\setminus B_{\lambda_nR}(x_n)} |\psi_n|^4 dx = 0.
  \end{equation}

\

Under one bubble's assumption, we firstly have the following lemma.
\begin{lem}\label{lem:04}
  For any $\epsilon>0$, there exist $\delta_0=\delta_0(\epsilon)$, $R_0=R_0(\epsilon)$ such that for any $\delta\in (0,\delta_0)$, $R>R_0$, there exists $N=N(\epsilon,\delta,R)>0$ such that when $n\ge N$, there holds
\begin{equation*}
  \sup_{t\in [\lambda_nR,\delta]}\int_{B_{2t}(x_n)\setminus B_t(x_n)}( e^{2u_n} + |\psi_n|^4)dx <\epsilon,
\end{equation*} i.e.
\begin{equation}\label{equat:09}
 \lim_{R\rightarrow 0}\lim_{\delta \rightarrow 0}\lim_{n\rightarrow \infty} \sup_{t\in [\lambda_nR,\delta]}\int_{B_{2t}(x_n)\setminus B_t(x_n)}( e^{2u_n} + |\psi_n|^4)dx =0.
\end{equation}
\end{lem}

\begin{proof}
  Similar to harmonic maps \cite{DingWeiyueandTiangang} or  super-Liouville equation \cite{Jost-Wang-Zhou-Zhu-1}, we first have following two facts.

\textbf{Fact 1:} For any $\epsilon>0$, there exists $R_0=R_0(\epsilon)$ such that for any $R>R_0$ and $T>100$, there exists $N(\epsilon,R)$ such that for any $n\ge N$, there holds
\begin{equation*}
  \int_{B_{\lambda_nRT}(x_n)\setminus B_{\lambda_nR}(x_n)} (e^{2u_n} +|\psi_n|^4)dx <\epsilon.
\end{equation*}

\

\textbf{Fact 2:} For any $\epsilon>0$, there exists $\delta_0=\delta_0(\epsilon)$ such that for any $\delta<\delta_0$ and $T>100$, there exists $N(\epsilon,\delta)$ such that for any $n\ge N$, there holds
\begin{equation*}
  \int_{B_{\delta}(x_n)\setminus B_{\delta T^{-1}}(x_n)} (e^{2u_n} +|\psi_n|^4)dx <\epsilon.
\end{equation*}

Now we prove the lemma by a  contradiction argument. If not, then there exist $\epsilon_0>0$, $\delta>0,\ R>100$ and sequence $t_n\in [\lambda_nR,\delta]$ such that \begin{equation*}
  \int_{B_{2t_n}(x_n)\setminus B_{t_n}(x_n)} (e^{2u_n} + |\psi_n|^4 )dx\ge \epsilon_0.
\end{equation*}
By above two facts, we have $$\lim_{n\to\infty}\frac{t_n}{\lambda_nR}=\lim_{n\to\infty}\frac{\delta}{t_n}=+\infty.$$ Denoting
$$(v_n(x), \phi_n(x)):=(u_n(x_n+t_nx)+\ln t_n, \sqrt{t_n}\psi_n(x_n+t_nx)),$$ then for all $n$, we have
\begin{equation}\label{inequ:07}
  \int_{B_2(0)\setminus B_1(0)} (e^{2v_n} + |\phi_n|^4 )dx\ge \epsilon_0.
\end{equation}
Moreover, for $\forall\ T>100$, when $n$ is big, $(v_n,\phi_n)$ satisfies
\begin{equation*}
  \begin{cases}
    -\Delta v_n = 2e^{2v_n} - e^{v_n}|\phi_n|^2,\\
    \slashed{D}\phi_n = -(e^{v_n} + 2F(x_n+t_nx)|\phi_n|^2)\phi_n,
  \end{cases}\ \ x \in \ \  B_{T}(0)\setminus B_{T^{-1}}(0).
\end{equation*}

\

There are three possible cases.

\

\textbf{Case 1:} There exists $T>0$ such that $(v_n,\phi_n)$ has a energy concentration point $q\in  B_{T}(0)\setminus B_{T^{-1}}(0)$, i.e.
\begin{eqnarray}
  \lim_{n\rightarrow \infty}\int_{B_r(q)}( e^{2v_n} + |\phi_n|^4 )dx\ge \epsilon_0>0,\ \ \forall \ r>0.
\end{eqnarray}

\

In this case, similar to construct the first bubble, we will get the second bubble, this is a contradiction.

\

If Case 1 is false, then for any $T>0$, $(v_n,\phi_n)$ has no energy concentration point in $ B_{T}(0)\setminus B_{T^{-1}}(0)$. By small energy regularity Lemma \ref{lem:small-energy-regul}, we have to consider following two cases.

\

\textbf{Case 2:}  $v_n \rightarrow -\infty$ uniformly in any compact subset of $\R^2\setminus \{0\}$.

\

In this case, we have $\phi_n\to \phi$ in $C^1_{loc}(\R^2\setminus \{0\})$  where
\begin{equation*}
  \slashed{D} \phi = 2\mu |\phi|^2 \phi,\ \ \ x\in \R^2\setminus \{0\}.
\end{equation*}
By \eqref{inequ:07}, we get
\begin{equation*}
  \int_{B_2(0)\setminus B_1(0)}  |\phi|^4 dx\ge \epsilon_0,\ \ \int_{\R^2}  |\phi|^4 dx\leq C.
\end{equation*}
By conformal invariance and removable singularity (Theorem 5.1 in \cite{Jost-Liu-Zhu}), we can see $\phi$ as a spinor from $S^2$, i.e.
\begin{equation*}
  \slashed{D} \phi = 2\mu |\phi|^2 \phi,\ \ \ x\in S^2.
\end{equation*} This is a contradiction since we get the second bubble.

\

\textbf{Case 3:} $(v_n,\phi_n)\to (v,\phi)$ in $C^2_{loc}(\R^2\setminus \{0\})$ and $(v,\phi)$ satisfies
\begin{equation*}
  \begin{cases}
  -\Delta v &= 2e^{2v} - e^v |\phi|^2, \\
  \slashed{D} \phi  &=-(e^v+2\mu|\phi|^2)\phi,
  \end{cases} \ \ in\ \ \R^2\setminus \{0\},
\end{equation*} and $$\int_{B_2(0)\setminus B_1(0)}(e^{2v}+  |\phi|^4) dx\ge \epsilon_0,\ \ \int_{\R^2}(e^{2v}+  |\phi|^4) dx\le C.$$
There are two singularities: $+\infty$ and $0$. Next, we will prove this two singularities are all removable. In fact, we just need to show singularity at $0$ is removable, then singularity at $+\infty$ is also removable from Corollary \ref{removability at infinity}.

Firstly, by Proposition \ref{Removability of local singularity}, we know that $$v(x)=-\frac{\gamma}{2\pi}+O(1),\ \ \nabla v(x)=-\frac{\gamma}{2\pi}\frac{x}{|x|^2}+O(|x|^{-1+\epsilon'})\ \ near\ \ 0$$ for some small constant $\epsilon'>0$, where $\gamma<2\pi$ is a constant.

For any $s>0$, integrating by parts, we get
\begin{align*}
\gamma=-\lim_{s\to 0}\lim_{n\to\infty}\int_{\partial B_s(0)}\frac{\partial v_n}{\partial r}=
\lim_{s\to 0}\lim_{n\to\infty}\int_{B_s(0)}(-\Delta v_n)dx=\lim_{s\to 0}\lim_{n\to\infty}\int_{B_s(0)}\left(2e^{2v_n} - e^{v_n}|\phi_n|^2 \right)dx.
\end{align*}

Secondly, since $(u_n,\psi_n)$ satisfy  Pohozaev identity that
\begin{equation*}
  \begin{aligned}
    0&=C(u_n,\psi_n)\\&=
    st_n\int _{\partial B_{st_n}(x_n)} \left(\left|
      \frac{\partial u_n}{\partial |x-x_n|}
    \right|^2 - \frac{1}{2}|\nabla u_n|^2\right)
    - \int _{B_{st_n}(x_n)} \left(2e^{2u_n} - e^{u_n}|\psi_n|^2 -|\psi_n|^4 (x-x_n)\cdot \nabla F\right)dx\\
    &\quad +{st_n}\int_{\partial B_{st_n}(x_n)} \left(e^{2u_n} + F|\psi_n|^4\right)\\
    &\quad -\frac{1}{2}\int_{\partial B_{st_n}(x_n)}\left( \left<
      (x-x_n)\cdot \psi_n,\frac{\partial \psi_n}{\partial |x-x_n|}
    \right>+\left<
      \frac{\partial \psi_n}{\partial |x-x_n|},(x-x_n)\cdot \psi_n
    \right>\right)\\
    &=s\int _{\partial B_{s}(0)} \left(\left|
      \frac{\partial v_n}{\partial |x|}
    \right|^2 - \frac{1}{2}|\nabla v_n|^2\right)
    - \int _{B_{s}(0)} \left(2e^{2v_n} - e^{v_n}|\phi_n|^2 -|\phi_n|^4 x\cdot \nabla F(x_n+t_nx)\right)dx\\
    &\quad +{s}\int_{\partial B_{s}(0)} \left(e^{2v_n} + F( x_n+t_nx)|\phi_n|^4\right) -\frac{1}{2}\int_{\partial B_{s}(0)}\left( \left<
      x\cdot \phi_n,\frac{\partial \phi_n}{\partial |x|}
    \right>+\left<
      \frac{\partial \phi_n}{\partial |x|},x\cdot \phi_n
    \right>\right),
\end{aligned}
\end{equation*}
taking the limit that $\lim_{s\to 0}\lim_{n\to\infty}$ and noting that $$\int _{B_{s}(0)} \left||\phi_n|^4 x\cdot \nabla F(x_n+t_nx)\right|dx\leq Cs\int _{B_{s}(0)} |\phi_n|^4 dx\leq Cs,$$ we obtain
\begin{align*}
  0&=C(u_n,\psi_n)\\
  &=\lim_{s\to 0}s\int _{\partial B_{s}(0)} \left(\left|
      \frac{\partial v}{\partial |x|}
    \right|^2 - \frac{1}{2}|\nabla v|^2\right)
    - \lim_{s\to 0}\lim_{n\to\infty}\int _{B_{s}(0)} \left(2e^{2v_n} - e^{v_n}|\phi_n|^2 \right)dx\\
    &\quad +\lim_{s\to 0}s\int_{\partial B_{s}(0)} \left(e^{2v} + \mu|\phi|^4\right) -\lim_{s\to 0}\frac{1}{2}\int_{\partial B_{s}(0)}\left( \left<
      x\cdot \phi,\frac{\partial \phi}{\partial |x|}
    \right>+\left<
      \frac{\partial \phi}{\partial |x|},x\cdot \phi
    \right>\right).
\end{align*}

By Lemma \ref{lem:pohozaev-const} with $ F\equiv\mu$, the Pohozaev constant for $(v,\phi)$ is
  \begin{align*}
    C(v,\phi)=\lim_{s\to 0}C(v,\phi;s)
    &=\lim_{s\to 0}s\int _{\partial B_s} \left(\left|
      \frac{\partial v}{\partial \nu}
    \right|^2- \frac{1}{2}|\nabla v|^2\right)- \lim_{s\to 0}\int _{B_s} \left(2e^{2v} - e^v|\phi|^2 \right)dx
    \\
    &\quad +\lim_{s\to 0}s\int_{\partial B_s}\left( e^{2v} + \mu|\phi|^4\right)
     -\frac{1}{2}\lim_{s\to 0}\int_{\partial B_s} \left(\left<
      x\cdot \phi,\frac{\partial \phi}{\partial \nu}
    \right>+\left<
      \frac{\partial \phi}{\partial \nu},x\cdot \phi
    \right>\right)\\
    &= \lim_{s\to 0}s\int _{\partial B_s} \left(\left|
      \frac{\partial v}{\partial \nu}
    \right|^2- \frac{1}{2}|\nabla v|^2\right)
    +\lim_{s\to 0}s\int_{\partial B_s}\left( e^{2v} + \mu|\phi|^4\right)
     \\&\quad -\frac{1}{2}\lim_{s\to 0}\int_{\partial B_s} \left(\left<
      x\cdot \phi,\frac{\partial \phi}{\partial \nu}
    \right>+\left<
      \frac{\partial \phi}{\partial \nu},x\cdot \phi
    \right>\right).
    \end{align*}

Thus, we arrive at $$0=C(u_n,\psi_n)=C(v,\phi)-\gamma.$$ By Proposition  \ref{Removability of local singularity}, we have $$C(v,\phi)=\gamma=\frac{\gamma^2}{4\pi},$$ which yields $\gamma=0$ and $C(v,\phi)=0$. Then for $(v,\phi)$, the singularity at $0$ is removable. By Corollary \ref{removability at infinity}, we get the second bubble which is a contradiction. We proved the lemma.

\end{proof}

Now we give the proof of the energy identity for the spinor.

\textbf{Proof of Theorem \ref{thm:main-2}:}
Similar to Dirac-harmonic maps \cite{zhao2007energy} or super-Liouville equation \cite{Jost-Wang-Zhou-Zhu-1}, we can divide $B_\delta(x_n)\setminus B_{\lambda_n R}(x_n)$ into finitely many parts,
\begin{equation*}
 B_\delta(x_n)\setminus B_{\lambda_n R}(x_n)= \bigcup_{k=1}^{N_n} P^k,\,\,P^k = B_{T^k}(x_n)\setminus B_{T^{k-1}}(x_n),\,\,T^0 = \lambda_nR<T^1....<T^{N_n} = \delta,
\end{equation*}
such that $N_n\le N_0$, where $N_0$ is a uniform integer for all large $n$, and
\begin{equation*}
  \int_{P^k} (e^{2u_n} + |\psi|^4)dx \le \frac{1}{4\Lambda^2},
\end{equation*} where $\Lambda$ is the constant in Lemma \ref{lem:07}.

By Lemma \ref{lem:07} and Lemma \ref{lem:04}, for any $\epsilon>0$, taking $R$ and $\delta^{-1}$ big, when $n$ is large enough, for each $l=1,...,N_n$, we have
\begin{equation*}
  \begin{aligned}
    \left(
     \int_{P^l} |\psi_n|^4dx
    \right)^{\frac{1}{4}}
    &\le \Lambda \left(
      \int_{B_{2T^l}(x_n)\setminus B_{\frac{1}{2}T^{l-1}}(x_n) } (e^{2u_n} +|\psi_n|^4)dx
    \right)^{\frac{1}{2}}\cdot
    \left(
      \int_{B_{2T^l}(x_n)\setminus B_{\frac{1}{2}T^{l-1}}(x_n) }|\psi_n|^4dx
    \right)^{\frac{1}{4}}\\
    &\quad +C\left(
    \int_{B_{2T^l}(x_n)\setminus B_{T^{l}}(x_n) }|\psi_n|^4dx
    \right)^{\frac{1}{4}}+
    C\left(
      \int_{B_{T^{l-1}}(x_n)\setminus B_{\frac{1}{2}T^{l-1}}(x_n) } |\psi_n|^4dx
    \right)^{\frac{1}{4}}\\
    &\le \Lambda \left(\left(
      \int_{P^l} (e^{2u_n} + |\psi_n|^4)dx
    \right)^{\frac{1}{2}}+ \epsilon^{\frac{1}{2}}\right)\cdot \left(
    \left(
      \int_{P^l} |\psi_n|^4dx
    \right)^{\frac{1}{4}} + \epsilon^{\frac{1}{4}}
    \right)+C\epsilon^{\frac{1}{4}}\\
    &\le \Lambda\left(
      \int_{P^l} (e^{2u_n} + |\psi_n|^4)dx
    \right)^{\frac{1}{2}}\cdot \left(
      \int_{P^l} |\psi_n|^4dx
    \right)^{\frac{1}{4}} +C\epsilon^{\frac{1}{4}}\\
    &\le \frac{1}{2}  \left(
     \int_{P^l} |\psi_n|^4dx
    \right)^{\frac{1}{4}} + C\epsilon^{\frac{1}{4}},\\
  \end{aligned}
\end{equation*}
which implies
\begin{equation*}
  \int_{P^l} |\psi_n|^4dx \le C\epsilon,\ \ and\ \ \int_{B_\delta(x_n)\setminus B_{\lambda_n R}(x_n)} |\psi_n|^4dx \le CN_0\epsilon.
\end{equation*}
This implies \eqref{equat:23} and we complete the proof of Theorem \ref{thm:main-2}.

\

\section{Blow-up values and energy identity for the function}\label{sec:Blow-up values and energy identity for the function}

\

In this section, we will firstly use the energy identity of the spinor to rule out the partial case in the conclusion $(c)$ of Theorem \ref{thm:main-1}, see Theorem \ref{thm:main-3}. Secondly, we will calculate the blow-up masses at a blow-up point, where there may appear two different values according to two kinds of singularities, see Theorem \ref{thm:main-4}. Lastly, we will show the energy identity for the function, i.e. the proof of Theorem \ref{thm:main-5}.

\

As an application of the energy identity for the spinor, we firstly show the proof for Theorem \ref{thm:main-3}.

\begin{proof}[\textbf{Proof of Theorem \ref{thm:main-3}:}]

  We use a contradiction argument. If not, then $(u_n,\psi_n)$ is bounded in $L^\infty_{loc} (M\setminus \Sigma)$.
  By the standard elliptic estimate we can choose a subsequence of $(u_n,\psi_n)$ such that $(u_n,\psi_n)$ converges to $(u,\psi)$ in $C^\infty_{loc}(M\setminus S)$.
  We choose a fixed blow up point $p$ and $R>0$ such that $p$ is the only blow-up point in $B_R(p)$.
  Thus $(u,\psi)$ satisfy
  \begin{equation*}
    \begin{cases}
          -\Delta u &= 2e^{2u} - e^u |\psi|^2, \\
        \slashed{D} \psi &= -(e^u + 2F|\psi|^2) \psi,
    \end{cases}\ \ in \ \ B_R(p)\setminus \{p\}.
  \end{equation*}
 By Proposition \ref{Removability of local singularity} and the proof of Lemma \ref{lem:04}, we know that $$u(x)=-\frac{\gamma}{2\pi}\ln |x-p|+O(1),\ \ \nabla u(x)=-\frac{\gamma}{2\pi}\frac{x}{|x|^2}+O(|x|^{-1+\epsilon'})\ \ near\ \ 0$$ for some small constant $\epsilon'>0$, where $\gamma<2\pi$ is a constant defined by $$\gamma=\lim_{\delta\rightarrow 0}\lim_{n\rightarrow \infty} \int_{B_\delta(p) } \left(2e^{2u_n} - e^{u_n} |\psi_n|^2 \right)dx.$$

 Without loss of generality, we assume there is only one bubble at $p$. Using the notation in Lemma \ref{lem:04}, we have
  \begin{align*}
    \gamma&=\lim_{\delta\rightarrow 0}\lim_{n\rightarrow \infty}\int_{B_\delta(p) } \left(2e^{2u_n} - e^{u_n} |\psi_n| ^2 \right)dx\\ &=\lim_{R\rightarrow \infty}\lim_{\delta\rightarrow 0}\lim_{n\rightarrow \infty}\int_{B_\delta(p)\setminus B_{\lambda_nR}(x_n) } \left(2e^{2u_n} - e^{u_n} |\psi_n| ^2 \right)dx+\lim_{R\rightarrow \infty}\lim_{\delta\rightarrow 0}\lim_{n\rightarrow \infty}\int_{B_{\lambda_nR}(x_n) } \left(2e^{2u_n} - e^{u_n} |\psi_n| ^2 \right)dx\\
    &\geq \lim_{R\rightarrow \infty}\lim_{\delta\rightarrow 0}\lim_{n\rightarrow \infty}\int_{B_\delta(p)\setminus B_{\lambda_nR}(x_n) } \left( - e^{u_n} |\psi_n| ^2 \right)dx+\lim_{R\rightarrow \infty}\lim_{\delta\rightarrow 0}\lim_{n\rightarrow \infty}\int_{B_{\lambda_nR}(x_n) } \left(2e^{2u_n} - e^{u_n} |\psi_n| ^2 \right)dx=4\pi.
  \end{align*}
  which is a contradiction since $\gamma<2\pi$.

\end{proof}

Now we define the following blow-up value at a blow-up point $p\in\Sigma$ that
\begin{equation*}
  m(p):=\lim_{\delta\rightarrow 0}\lim_{n\rightarrow 0} \int_{B_\delta(p)} \left(2e^{2u_n} - e^{u_n}|\psi_n|^2\right) dx.
\end{equation*}
In order to calculate the blow-up value, we need the following lemma.

\begin{lem}\label{lem:Green funciton for u_n}
  Let $(u_n,\psi_n) $ be a sequence solution to \eqref{equat:01}  with finite energy $E(u_n,\psi_n;M)<C$. If $\Sigma_3\neq \emptyset$, then there  exists $G\in W^{1,p}(G)\cap C^2_{loc}(M\setminus \Sigma)$ with $\int_M Gdx = 0$ for $1<q<2$ such that
  \begin{equation*}
    u_n - \frac{1}{|M|}\int_M u_ndx \rightarrow G
  \end{equation*}
  in $C^2_{loc}(M\setminus \Sigma)$ and weakly in  $W^{1,q}(M)$.
  Moreover, in $\Sigma = \{p_1,\cdots,p_l\}$ then for $R>0$ small such that
  $B_R(p_k)\cap S = \{p_k\}$, we have
  \begin{equation*}
    G=-\frac{1}{2\pi} m(p_k)\log|x-p_k| + g(x),
  \end{equation*}
  where $g\in C^2(B_R(p_k))$.
\end{lem}

\begin{proof}
  Let $p=\frac{q}{q-1}>2$, We have
  \begin{equation*}
    \norm{\nabla u_n}_{L^q(M)} \le \sup\left\{
      \left|
        \int_M \nabla u_n \nabla \phi
      \right|:\phi\in W^{1,p}(M),\int_M \phi =0,\norm{\phi}_{W^{1,p}(M)} =1
    \right\}.
  \end{equation*}
  By sobolev imbedding we have $    \norm{\phi}_{L^\infty (M)}\le C.$
  It is clear that
  \begin{equation*}
    \left|
      \int_M \nabla u_n \nabla \phi
    \right|\le \left|
      \int_M (\Delta u_n) \phi
    \right| \le C
    \left|
      \int_M 2e^{2u_n} + e^{u_n} |\psi_n|^2 + |\psi_n|^4
    \right|\le C.
  \end{equation*}
  Thus $ u_n - \bar{u}_n$ is bounded in $W^{1,q}(M)$,
  we assume that $u_n - \bar{u}_n$  weakly converges to $H$.
  Next, we define
  \begin{equation*}
    \begin{cases}
    -\Delta G = \sum_{p\in S} m(p)\delta_p - K_g,\\
      \int_M G=0.
    \end{cases}
  \end{equation*}

  For all $\phi\in C^\infty (M)$ we have
  \begin{align*}
    \lim_{n\to\infty}\int_M \nabla(u_n - G) \nabla \phi dx
     &=  \lim_{n\to\infty}\int_M \left(2e^{2u_n} - e^{u_n}|\psi_n|^2  - \sum_{p\in \Sigma} m(p)\delta_p\right) \phi dx\\
     &= \lim_{\delta\to 0} \lim_{n\to\infty}\int_{M\setminus \cup_{p\in\Sigma}B_\delta(p)} \left(2e^{2u_n} - e^{u_n}|\psi_n|^2 \right) \phi dx\\
     &\quad + \lim_{\delta\to 0} \lim_{n\to\infty}\sum_{p\in\Sigma}\int_{B_\delta(p)} \left(2e^{2u_n} - e^{u_n}|\psi_n|^2 - \sum_{p\in \Sigma} m(p)\delta_p\right) \phi dx\\
     &= \lim_{\delta\to 0} \lim_{n\to\infty}\sum_{p\in\Sigma}\int_{B_\delta(p)} \left(2e^{2u_n} - e^{u_n}|\psi_n|^2- \sum_{p\in \Sigma} m(p)\delta_p\right) \phi dx=0.
  \end{align*}

  This means $\nabla u_n \rightarrow \nabla G$ in $ L^q (M)$
  and $\nabla G = \nabla H$ almost everywhere.
  Then $G = C+ H$ where $C$ is a constant.
  For $u_n - \bar{u}_n \rightarrow H$ in $L^q$,
  \begin{equation*}
    0=\int_M Gdx = \int_M Hdx + \int_M Cdx = C\cdot |M|.
  \end{equation*}
  Then $G=H$ almost everywhere.
\end{proof}

\

With the help of above lemma, we compute the blow-up values.

\begin{proof}[\textbf{Proof of Theorem \ref{thm:main-4}:}]
  Assume that $p$ is the only blow up point in $B_R(p)$. For simplicity of notations, we assume $p=0$. Recall the Pohozaev identity for $(u_n,\psi_n)$ that
\begin{equation*}
  \begin{aligned}
    0&=
    s\int _{\partial B_{s}(0)} \left(\left|
      \frac{\partial u_n}{\partial r}
    \right|^2 - \frac{1}{2}|\nabla u_n|^2\right)
    - \int _{B_{s}(0)} \left(2e^{2u_n} - e^{u_n}|\psi_n|^2 -|\psi_n|^4 x\cdot \nabla F\right)dx\\
    &\quad +{s}\int_{\partial B_{s}(0)} \left(e^{2u_n} + F|\psi_n|^4\right)-\frac{1}{2}\int_{\partial B_{s}(0)}\left( \left<
      x\cdot \psi_n,\frac{\partial \psi_n}{\partial r}
    \right>+\left<
      \frac{\partial \psi_n}{\partial r},x\cdot \psi_n
    \right>\right).
\end{aligned}
\end{equation*}
By Lemma \ref{lem:Green funciton for u_n}, we first have
  \begin{equation*}
    \lim_{s\rightarrow 0}\lim_{n\rightarrow \infty} s\int _{\partial B_s}\left( \left|
      \frac{\partial u_n}{\partial \nu}
    \right|^2 - \frac{1}{2}|\nabla u_n|^2\right) = \frac{m^2(p)}{4\pi}.
  \end{equation*}
Secondly, noting that
  \begin{align*}
    &\lim_{n\to\infty}\left\{s\int_{\partial B_{s}(0)} \left(e^{2u_n} + F|\psi_n|^4\right)-\frac{1}{2}\int_{\partial B_{s}(0)}\left( \left<
      x\cdot \psi_n,\frac{\partial \psi_n}{\partial r}
    \right>+\left<
      \frac{\partial \psi_n}{\partial r},x\cdot \psi_n
    \right>\right)\right\}\\
    &=s\int_{\partial B_{s}(0)} \left( F|\psi|^4\right)-\frac{1}{2}\int_{\partial B_{s}(0)}\left( \left<
      x\cdot \psi,\frac{\partial \psi}{\partial r}
    \right>+\left<
      \frac{\partial \psi}{\partial r},x\cdot \psi
    \right>\right).
  \end{align*}
    Since  $$F(x)  |\psi|^4 -\frac{1}{2} \frac{1}{|x|}\left( \left<
      x\cdot \psi,\frac{\partial \psi}{\partial r}
    \right>+\left<
      \frac{\partial \psi}{\partial r},x\cdot \psi
    \right>\right)\in L^1(B_s(0)),$$
    by Fubini's theorem, there exists $\{s_k \}_{k}$ with $\lim_{k\to\infty}s_k= 0$ such that
    \begin{equation*}
      \lim_{k\to\infty}\lim_{n\to\infty}\left\{s_k\int_{\partial B_{s_k}(0)} \left(e^{2u_n} + F|\psi_n|^4\right)-\frac{1}{2}\int_{\partial B_{s_k}(0)}\left( \left<
      x\cdot \psi_n,\frac{\partial \psi_n}{\partial r}
    \right>+\left<
      \frac{\partial \psi_n}{\partial r},x\cdot \psi_n
    \right>\right)\right\}=0.
    \end{equation*}
    Choosing  $s=s_k$ in the Pohozaev identity and  then taking the limit $\lim_{k\to\infty}\lim_{n\to\infty}$, we get
    \begin{equation*}
      \frac{m^2(p)}{4\pi} = m(p),
    \end{equation*} which immediately implies that either $m(p)=4\pi$ or $0$. Then the conclusions follows from Theorem \ref{thm:main-3}.
\end{proof}

\

Next we prove the energy identity for the function.

\

\begin{lem}\label{lem:05}
Suppose $\Sigma_3\neq \emptyset$. Then for each $x_i\in\Sigma_3\subset\Sigma$, among the bubbles $(u^{i,k},\psi^{i,k})$, $k=1,...,L_i$ in Theorem \ref{thm:main-2}, there is only one super-Liouville type bubble (w.l.o.g, denoted by $u^{i,1}$) where the other bubbles are all spinorial Yamabe type bubble. Moreover, we have following energy identity $$\lim_{\delta\to 0}\lim_{n\to\infty}\int_{B_\delta(x_i)}e^{2u_n}dx=\int_{S^2}e^{2u^{i,1}}dx.$$
\end{lem}
\begin{proof}

For simplicity of notations and expressing the idea of the proof better, without loss of generality, we assume there are at most two bubbles at a considered blow-up point $p\in\Sigma$. The general case (finite bubbles) can be derived by an  induction argument.

For $p\in\Sigma_3$, by the process of constructing the first bubble at the beginning of Section \ref{sec:energy-identity-spinor}, we know that there exist $x_n\to p$ and $\lambda_n\to 0$ such that   \begin{equation*}
(\tilde{u}_n(x), \tilde{\psi}_n(x)) :=\left( u_n (\lambda_n x+ x_n) + \log\lambda_n, \lambda_n^{\frac{1}{2}} \psi_n (\lambda_n x + x_n)\right)\to (\tilde{u},\tilde{\psi})\ \ in\ \ C^2_{loc}(\R^2),\end{equation*} where $(\tilde{u},\tilde{\psi})$ is a super-Liouville type bubble.

Next, we need to show that if there are additional bubbles, they must be spinorial Yamabe type bubble. Moreover, the following equality hods
\begin{equation}\label{equat:11}
\lim_{\delta\to 0}\lim_{R\to\infty}\lim_{n\to\infty}\int_{B_{\delta}(x_n)\setminus B_{\lambda_nR}(x_n)}e^{2u_n}dx=0.
\end{equation}

\

Firstly, integrating by parts, by Lemma \ref{lem:Green funciton for u_n} and Proposition \ref{prop:01}, we have
\begin{align}\label{equat:13}
\lim_{\delta\to 0}\lim_{R\to\infty}\lim_{n\to\infty}\int_{B_{\delta}(x_n)\setminus B_{\lambda_nR}(x_n)}(2e^{2u_n}-e^{u_n}|\psi_n|^2)dx&=\lim_{\delta\to 0}\lim_{R\to\infty}\lim_{n\to\infty}\int_{B_{\delta}(x_n)\setminus B_{\lambda_nR}(x_n)}(-\Delta u_n)dx\notag\\
&=-\lim_{\delta\to 0}\lim_{n\to\infty}\int_{\partial B_{\delta}(p)}\frac{\partial u_n}{\partial r}+\lim_{R\to\infty}\lim_{n\to\infty}\int_{\partial B_{ R}(0)}\frac{\partial \tilde{u}_n}{\partial r} =0.
\end{align}

We have to consider  following two cases.

\

\textbf{Case I:} There is only one bubble at $p$.

\

In this case, by \eqref{equat:13} and energy identity for the spinor, we get
\begin{align*}
\lim_{\delta\to 0}\lim_{R\to\infty}\lim_{n\to\infty}\int_{B_{\delta}(x_n)\setminus B_{\lambda_nR}(x_n)}e^{2u_n}dx&=\lim_{\delta\to 0}\lim_{R\to\infty}\lim_{n\to\infty}\int_{B_{\delta}(x_n)\setminus B_{\lambda_nR}(x_n)}(2e^{2u_n}-e^{u_n}|\psi_n|^2)dx \\&\quad +\lim_{\delta\to 0}\lim_{R\to\infty}\lim_{n\to\infty}\int_{B_{\delta}(x_n)\setminus B_{\lambda_nR}(x_n)}(e^{u_n}|\psi_n|^2)dx=0.
\end{align*}

\

\textbf{Case II:} There are two bubbles at $p$.

\

 By the proof of Lemma \ref{lem:04}, there exist $\epsilon_0>0$, $\delta>0,\ R>100$ and sequence $t_n\in [\lambda_nR,\delta]$ such that \begin{equation*}
  \int_{B_{2t_n}(x_n)\setminus B_{t_n}(x_n)} (e^{2u_n} + |\psi_n|^4 )dx\ge \epsilon_0,\ \ \lim_{n\to\infty}\frac{t_n}{\lambda_nR}=\lim_{n\to\infty}\frac{\delta}{t_n}=+\infty.
\end{equation*}
Denote
$$(v_n(x), \phi_n(x)):=(u_n(x_n+t_nx)+\ln t_n, \sqrt{t_n}\psi_n(x_n+t_nx)).$$

Since $m(p)=4\pi$, the second bubble must be Yamabe-type bubble. Otherwise, $m(p)\geq 8\pi$. Thus, the $Case\ 3 $ in Lemma \ref{lem:04} will not happen and then  we have to consider the other two cases.

\

\textbf{Case 1 in Lemma \ref{lem:04}:} There exists $T>100$ such that $(v_n,\phi_n)$ has a energy concentration point $q\in  B_{T}(0)\setminus B_{T^{-1}}(0)$, i.e.
\begin{eqnarray}
  \lim_{n\rightarrow \infty}\int_{B_r(q)}( e^{2v_n} + |\phi_n|^4 )dx\ge \epsilon_0>0,\ \ \forall \ r>0.
\end{eqnarray}

In this case, we know that for sequence $(v_n,\phi_n)$, the blow-up point $q$ must be second type singularity. Moreover, there exist $y_n\to q$, $\mu_n\to 0$ such that
\begin{align*}
\tilde{v}_n(x):=v_n(y_n+\mu_nx)+\ln\mu_n\to -\infty\ \ & \mbox{ uniformly in any compact subset of }  \R^2,\\
\tilde{\phi}_n(x)):= \sqrt{\mu_n}\phi_n(y_n+\mu_nx))\to \tilde{\phi}(x)\ \ &in\ \ C^2_{loc}(\R^2),
\end{align*} where $\tilde{\phi}$ ia s spinorial Yamabe type bubble which is the second bubble. Since there is no third bubble, then we have
\begin{equation*}
\lim_{R\to +\infty}\lim_{\delta\to 0}\lim_{n\to\infty}  \sup_{t\in [\mu_nR,\delta]}\int_{B_{2t}(y_n)\setminus B_t(y_n)}( e^{2v_n} + |\phi_n|^4)dx =0,
\end{equation*} which implies the energy identity
\begin{equation}\label{equat:24}
\lim_{R\to +\infty}\lim_{\delta\to 0}\lim_{n\to\infty}\int_{B_\delta(q)\setminus B_{\mu_nR}(y_n)}|\phi_n|^4dx=0.
\end{equation}

At the same time, by small energy regularity theory, we have one of the following alternatives holds:
\begin{itemize}
\item[(1-1)] $v_n(x)\to -\infty$ uniformly in any compact subset of $\R^2\setminus \{0,q\}$ and $\phi_n\to\phi$ in $C^1_{loc}(\R^2\setminus \{0,q\})$ where $$\slashed{D}\phi=\mu |\phi|^2\phi,\ \ x\in \R^2\setminus \{0,q\}.$$
\item[(1-2)]  $(v_n,\phi_n)\to (v,\phi)$ in $C^2_{loc}(\R^2\setminus \{0,q\})$ where $(v,\phi)$ satisfies
\begin{equation*}
  \begin{cases}
  -\Delta v &= 2e^{2v} - e^v |\phi|^2, \\
  \slashed{D} \phi  &=-(e^v+2\mu|\phi|^2)\phi,
  \end{cases} \ \ in\ \ \R^2\setminus \{0,q\}.
\end{equation*}
\end{itemize}

For Case $(1-2)$, similar to the argument as in the Case $3$ of Lemma \ref{lem:04}, we can show that the singularities $0$ and $q$ are all removable. This gives us the third bubble which is a contradiction, i.e. Case $(1-2)$ will not happen. For Case $(1-1)$, $\phi$ must be trivial, i.e. $\phi\equiv 0$. Otherwise, by removable singularity, $\phi$ is the third bubble which is also a contradiction.

Thus, we have
\begin{align}\label{equat:18}
\lim_{n\to\infty}\int_{B_{Tt_n}(x_n)\setminus B_{T^{-1}t_n}(x_n)}e^{u_n}|\psi_n|^2dx
&=\lim_{n\to\infty}\int_{B_T(0)\setminus B_{T^{-1}}(0)}e^{v_n}|\phi_n|^2dx\notag\\&=\lim_{\delta\to 0} \lim_{n\to\infty}\int_{B_\delta(q)}e^{v_n}|\phi_n|^2dx\notag\\
&=\lim_{R\to\infty}\lim_{\delta\to 0} \lim_{n\to\infty}\int_{B_\delta(q)\setminus B_{\mu_n R}(y_n)}e^{v_n}|\phi_n|^2dx+ \lim_{R\to\infty}\lim_{n\to\infty}\int_{B_R(0)}e^{\tilde{v}_n}|\tilde{\psi}_n|^2dx=0,
\end{align} where the last equality follows from \eqref{equat:24}.

Since there is no third bubble, by the proof of Lemma \ref{lem:04}, we also have
\begin{equation*}
\lim_{\delta\to 0}\lim_{n\to\infty}  \sup_{t\in [Tt_n,\delta]}\int_{B_{2t}(x_n)\setminus B_t(x_n)}( e^{2u_n} + |\psi_n|^4)dx=\lim_{R\to +\infty}\lim_{n\to\infty}  \sup_{t\in [\lambda_nR,T^{-1}t_n]}\int_{B_{2t}(x_n)\setminus B_t(x_n)}( e^{2u_n} + |\psi_n|^4)dx =0,
\end{equation*} which implies the energy identity
\begin{equation}\label{equat:17}
\lim_{\delta\to 0}\lim_{n\to\infty}\int_{B_\delta(p)\setminus B_{Tt_n}(x_n)}|\psi_n|^4dx=\lim_{R\to +\infty}\lim_{n\to\infty}\int_{B_{T^{-1}t_n}(x_n)\setminus B_{\lambda_nR}(x_n)}|\psi_n|^4dx=0.
\end{equation}

By \eqref{equat:17} and \eqref{equat:18}, we have $$\lim_{\delta\to 0}\lim_{n\to\infty}\int_{B_\delta(p)\setminus B_{\lambda_nR}(x_n)}e^{u_n}|\psi_n|^2dx=0.$$ Then \eqref{equat:11} follows immediately from \eqref{equat:13}.

\

\textbf{Case 2 in Lemma \ref{lem:04}:} If $v_n \rightarrow -\infty$ uniformly on all compact subset of $\mathbb{R}^2 \setminus \{0\}$, we have $\phi_n \rightarrow \phi$ in $C^2_{loc}(\mathbb{R}^2\setminus \{0\})$, where $\phi$ is a Spinorial Yamabe type bubble which is the second bubble.

\

It is easy to see that
\begin{equation}\label{equat:15}
  \lim_{T\to\infty}\lim_{n\to\infty}\int_{B_{Tt_n}(x_n)\setminus B_{T^{-1}t_n}(x_n)}e^{u_n}|\psi_n|^2dx
  =\lim_{T\to\infty}\lim_{n\to\infty}\int_{B_T(0)\setminus B_{T^{-1}}(0)}e^{v_n}|\phi_n|^2dx = 0.
\end{equation}

Since there is no third bubble, similar to the proof of Lemma \ref{lem:04},  we  have
\begin{align*}
\lim_{T\to\infty}  \lim_{\delta\to 0}\lim_{n\to\infty}  \sup_{t\in [Tt_n,\delta]}\int_{B_{2t}(x_n)\setminus B_t(x_n)}( e^{2u_n} + |\psi_n|^4)dx&=0,\\ \lim_{T\to\infty}  \lim_{R\to +\infty}\lim_{n\to\infty}  \sup_{t\in [\lambda_nR,T^{-1}t_n]}\int_{B_{2t}(x_n)\setminus B_t(x_n)}( e^{2u_n} + |\psi_n|^4)dx &=0,
\end{align*} which implies the energy identity
\begin{equation}\label{equat:14}
\lim_{T\to\infty}\lim_{\delta\to 0}\lim_{n\to\infty}\int_{B_\delta(p)\setminus B_{Tt_n}(x_n)}|\psi_n|^4dx=\lim_{T\to\infty}\lim_{R\to +\infty}\lim_{n\to\infty}\int_{B_{T^{-1}t_n}(x_n)\setminus B_{\lambda_nR}(x_n)}|\psi_n|^4dx=0.
\end{equation}

By \eqref{equat:14} and \eqref{equat:15}, we have
\begin{align*}
&\lim_{R\to\infty}\lim_{\delta\to 0}\lim_{n\to\infty}\int_{B_\delta(p)\setminus B_{\lambda_nR}(x_n)}e^{u_n}|\psi_n|^2dx\\
&=\lim_{T\to\infty}\lim_{\delta\to 0}\lim_{n\to\infty}\int_{B_\delta(p)\setminus B_{Tt_n}(x_n)}e^{u_n}|\psi_n|^2dx+ \lim_{T\to\infty}\lim_{n\to\infty}\int_{B_{Tt_n}(x_n)\setminus B_{T^{-1}t_n}(x_n)}e^{u_n}|\psi_n|^2dx\\&\quad +\lim_{T\to\infty}\lim_{R\to +\infty}\lim_{n\to\infty}\int_{B_{T^{-1}t_n}(x_n)\setminus B_{\lambda_nR}(x_n)}e^{u_n}|\psi_n|^2dx=0.
 \end{align*} Then \eqref{equat:11} follows immediately from \eqref{equat:13}.

\end{proof}

\

\begin{lem}\label{lem:06}
Suppose $\Sigma_3\neq \emptyset$. Then for each $x_i\in\Sigma\setminus \Sigma_3$, the following two alternatives hold:
\begin{itemize}
\item[(1)] If $m(x_i)=4\pi$, then among the bubbles $(u^{i,k},\psi^{i,k})$, $k=1,...,L_i$ in Theorem \ref{thm:main-2}, there is only one super-Liouville type bubble (w.l.o.g, denoted by $u^{i,1}$) where the other bubbles are all spinorial Yamabe type bubble. Moreover, we have following energy identity $$\lim_{\delta\to 0}\lim_{n\to\infty}\int_{B_\delta(x_i)}e^{2u_n}dx=\int_{S^2}e^{2u^{i,1}}dx.$$
\item[(2)] If $m(x_i)=0$, then the bubbles $(u^{i,k},\psi^{i,k})$, $k=1,...,L_i$ in Theorem \ref{thm:main-2} are all spinorial Yamabe type bubble. Moreover, we have following energy identity $$\lim_{\delta\to 0}\lim_{n\to\infty}\int_{B_\delta(x_i)}e^{2u_n}dx=0.$$
\end{itemize}
\end{lem}

\

Before giving its proof, we first show following two lemmas.

\

For $p\in \Sigma$, taking $r_0\in (0,1)$ such that $p\in B_{r_0}(p)$ is the only blow-up point in $B_{r_0}(p)$, by Lemma \ref{lem:Green funciton for u_n}, we have $osc_{\partial B_{r_0}(p)}u_n\leq C$. Using the isothermal coordinates at $p$, then $(B_{r_0}(p),g)$ isometric to $(B_1(0),e^{2\xi(x)}((dx^1)^2+(dx^2)^2))$, where $\xi(x)$ is a smooth function with $\xi(0)=0$. Without loss of generality, we assume $\xi(x)\equiv 1$. Otherwise, we just need to consider its conformal transformation that $\tilde{u}_n(x)=u_n\circ f+\xi(x),\ \tilde{\psi}_n(x)=e^{\frac{1}{2}\xi(x)}\psi_n\circ f$ where $f$ is the isometric map.

\

Let $(x_n,\lambda_n)$ be the first bubble at $p$ constructed at the beginning of Section \ref{sec:energy-identity-spinor}. If there is no bubble in the neck domain, then we have the following lemma.

\begin{lem}\label{lem:01}
If \begin{equation}\label{equat:25}
\lim_{\delta\to 0} \lim_{R\to\infty} \lim_{n\to\infty}  \sup_{t\in [\lambda_nR,\delta]}\int_{B_{2t}(x_n)\setminus B_t(x_n)}( e^{2u_n} + |\psi_n|^4)dx =0,
\end{equation} then there exists a big constant $N_0>0$ such that for $\delta^{-1},R>N_0$, the following two alternatives hold:
\begin{itemize}
\item[(1)] For any $t_n\in [2\lambda_nR, \delta]$, $u_n(x)$ has fast decay on $\partial B_{t_n}(x_n)$, i.e. $$\lim_{n\to\infty}\max_{x\in\partial B_{t_n}(x_n)}\left(u_n(x)+\ln |x-x_n|\right)=-\infty.$$

\

\item[(2)]  $$osc_{B_{\frac{1}{2}d_n}(x)}u_n(y)\leq C,\ \ \forall\ \ x\in B_{\delta}(x_n)\setminus B_{8\lambda_nR}(x_n),$$ when $n$ is big enough, where $d_n:=|x-x_n|$.

 \end{itemize}
\end{lem}
\begin{proof}
We divide the proof into two steps according to two conclusions.

\

\textbf{Step 1:} By \eqref{equat:25}, for any small constant $\epsilon>0$, there  exist $\delta_0=\delta_0(\epsilon)>0$, $R_0=R_0(\epsilon)>0$ and $N_0=N_0(\epsilon, \delta, R)$ such that for any $\delta<\delta_0$, $R>R_0$, we have $$\sup_{t\in [\lambda_nR,\delta]}\int_{B_{2t}(x_n)\setminus B_t(x_n)}( e^{2u_n} + |\psi_n|^4)dx \leq \epsilon,$$ when $n>N_0$.

For any $2\lambda_nR\leq t_n\leq \frac{1}{2}\delta$, let $$(v_n(x), \phi_n(x)):=(u_n(x_n+t_nx)+\ln t_n, \sqrt{t_n}\psi_n(x_n+t_nx)).$$ Then it is easy to see that $(v_n,\phi_n)$ satisfies
\begin{equation*}
  \begin{cases}
    -\Delta v_n = 2e^{2v_n} - e^{v_n}|\phi_n|^2,\\
    \slashed{D}\phi_n = -(e^{v_n} + 2F(x_n+t_nx)|\phi_n|^2)\phi_n,
  \end{cases}\ \ x \in \ \  B_{2}(0)\setminus B_{2^{-1}}(0)
\end{equation*} and $$E(v_n,\phi_n;B_{2}(0)\setminus B_{2^{-1}}(0))\leq C\epsilon.$$ Taking $\epsilon$ small such that $C\epsilon<\epsilon_1$ where $\epsilon_1$ is the constant in Lemma \ref{lem:small-energy-regul}, then by Lemma \ref{lem:small-energy-regul} we have $$\max_{ B_{\frac{3}{2}}(0)\setminus B_{\frac{3}{4}}(0)}v_n(x)+\|\phi_n\|_{L^\infty(B_{\frac{3}{2}}(0)\setminus B_{\frac{3}{4}}(0))}\leq C.$$
We claim:
\begin{equation}\label{equat:27}
\lim_{n\to\infty}\max_{ B_{\frac{3}{2}}(0)\setminus B_{\frac{3}{4}}(0)}v_n(x)= -\infty.
\end{equation}
Otherwise, by Lemma \ref{lem:small-energy-regul}, there holds $\|v_n\|_{L^\infty(B_{\frac{3}{2}}(0)\setminus B_{\frac{3}{4}}(0))}\leq C$. Then a standard elliptic estimate yields $(v_n,\phi_n)$ subconverges to $(v,\phi)$ in $C^2(B_{\frac{5}{4}}(0)\setminus B_{\frac{7}{8}}(0))$, where $(v,\phi)$ satisfies Liouville type bubble's equation. Thus we have
\begin{align*}
\lim_{n\to\infty}\int_{B_{2t_n}(x_n)\setminus B_{t_n}(x_n)}( e^{2u_n} + |\psi_n|^4)dx&= \lim_{n\to\infty}\int_{B_{2}(0)\setminus B_1(0)}( e^{2v_n} + |\phi_n|^4)dx\\&\geq \int_{B_{\frac{5}{4}}(0)\setminus B_{\frac{7}{8}}(0)}( e^{2v_n} + |\phi_n|^4)dx\geq C>0,
\end{align*} which is a contradiction to \eqref{equat:25}. Thus \eqref{equat:27} holds which immediately implies that $u_n(s)$ has fast decay on $\partial B_{t_n}(x_n)$ for $2\lambda_nR\leq t_n\leq \frac{1}{2}\delta$. By Lemma \ref{lem:Green funciton for u_n}, we get $osc_{B_\delta(x_n)\setminus B_{\frac{1}{2}\delta}(x_n)}u_n\leq C$. Then it is easy to see the first conclusion holds.

\

\textbf{Step 2:} Let $\eta_n$ be the solution of
\begin{align*}
\begin{cases}
-\Delta \eta_n(x)=0,\ \ &in\ \ B_{1}(0),\\
\eta_n(x)=u_n(x),\ \ &on\ \ \partial B_{1}(0).
\end{cases}
\end{align*} It is clear that $$osc_{B_{1}(0)}\eta_n\leq osc_{\partial B_{1}(0)}u_n\leq C.$$

Let $w_n=u_n-\eta_n$ and $$G(x,y)=-\frac{1}{2\pi}\ln |x-y|+H(x,y)$$ be the Green's function on $B_{1}(0)$ with respect to the Dirichlet boundary, where $H(x,y)$ is a smooth harmonic function. Then $$w_n(x)=\int_{B_1(0)}G(x,y)(-\Delta w_n)(y)dy.$$

For $x\in B_{\frac{1}{8}\delta}(x_n)\setminus B_{8\lambda_nR}(x_n)$ and $x_1,x_2\in B_{\frac{1}{2}d_n}(x)$, then
\begin{align*}
w_n(x_1)-w_n(x_2)&=\int_{B_1}(G(x_1,y)-G(x_2,y))(-\Delta w_n(y))dy\\
&=\frac{1}{2\pi}\int_{B_1}\log\frac{|x_1-y|}{|x_2-y|}(\Delta w_n(y))dy+O(1).
\end{align*}
We divide the integral into two parts, i.e. $B_1=B_{\frac{3}{4}d_n}(x)\cup B_1\setminus B_{\frac{3}{4}d_n}(x)$. Noting that $$\bigg|\log\frac{|x_1-y|}{|x_2-y|}\bigg|\leq C,\ \ y\in B_1\setminus B_{\frac{3}{4}d_n}(x),$$ we have $$\bigg|\int_{B_1\setminus B_{\frac{3}{4}d_n}(x)}\log\frac{ |x_1-y|}{ |x_2-y|}(\Delta w_n(y))dy \,\bigg|\leq C.$$
A direct computation yields
\begin{align*}
  &\bigg|\int_{B_{\frac{3}{4}d_n}(x)}\log\frac{ |x_1-y|}{ |x_2-y|}(\Delta w_n(y))dy \,\bigg|\\&\leq C\int_{B_{\frac{3}{4}d_n}(x)}\bigg|\log\frac{ |x_1-y|}{ |x_2-y|}\bigg| (e^{2u_n(y)}+e^{u_n(y)}|\psi_n(y)|^2)dy\\
  &=C\int_{B_{\frac{3}{4}}(0)}\bigg| \log\frac{ |x_1-x-d_ny|}{ |x_2-x-d_ny|}\bigg| (e^{2u_n(x+d_ny)}+e^{u_n(x+d_ny)}|\psi_n(x+d_ny)|^2)(d_n)^2dy\\
  &=C\int_{B_{\frac{3}{4}}(0)}\bigg| \log\frac{ \big|\frac{x_1-x}{d_n}-y\big|}{ \big|\frac{x_2-x}{d_n}-y\big|}\bigg| (e^{2u_n(x+d_ny)}+e^{u_n(x+d_ny)}|\psi_n(x+d_ny)|^2)(d_n)^2dy.
\end{align*} Since  $$x+d_ny\in B_{\frac{1}{2}\delta}(x_n)\setminus B_{2\lambda_nR}(x_n),\ \ |x+d_ny-x_n|\geq \frac{1}{4}d_n,\ \ \forall\ y\in B_{\frac{3}{4}},$$ by the first conclusion, we have $$u_n(x+d_ny)\leq C-\ln |x+d_ny-x_n|\leq C-\ln d_n,\ \ \sqrt{d_n}|\psi_n(x+d_ny)|\leq C,\ \  \forall \ y\in B_{\frac{3}{4}}.$$ From  $$\big|\frac{x_1-x}{d_n}\big|\leq\frac{1}{2},\ \ \big|\frac{x_2-x}{d_n}\big|\leq\frac{1}{2},$$ it is easy to conclude that $$ \big|\int_{B_{\frac{3}{4}d_n}(x)}
\log\frac{ |x_1-y|}{ |x_2-y|}(\Delta w_n(y))dy \big|\leq C.$$ This implies $$|w_n(x_1)-w_n(x_2)|\leq C.$$ Then we have $$osc_{B_{\frac{1}{2}d_n}(x)}u_n(y)\leq C,\ \ \forall\ \ x\in B_{\frac{1}{8}\delta}(x_n)\setminus B_{8\lambda_nR}(x_n).$$

For $x\in B_{\delta}(x_n)\setminus B_{\frac{1}{8}\delta}(x_n)$, by Lemma \ref{lem:Green funciton for u_n}, we have
\begin{align*}
osc_{B_{\frac{1}{2}d_n}(x)}u_n(y)\leq osc_{B_{\frac{3}{2}\delta}(x_n)\setminus B_{\frac{1}{16}\delta}(x_n)}u_n(y)\leq C,
\end{align*} which immediately yields the second conclusion.
\end{proof}

\

\begin{lem}\label{lem:02}
Under assumptions and notations  of Lemma \ref{lem:01}, let $8\lambda_nR\leq s_n^1<s_n^2\leq \delta$ with $\lim_{n\to\infty}\frac{s_n^2}{s_n^1}=+\infty$ and $$\int_{B_{s_n^1}(x_n)}(2e^{2u_n}-e^{u_n}|\psi_n|^2)dx=2\pi b+o(1),$$ where $b>1$, then we have $$\lim_{n\to\infty}\int_{B_{s_n^2}(x_n)\setminus B_{s_n^1}(x_n)}e^{2u_n}dx=0.$$
\end{lem}
\begin{proof}
Suppose the conclusion is false. Then there exists $\delta_0>0$ such that \begin{equation}\label{equat:26}\int_{B_{s_n^2}(x_n)\setminus B_{s_n^1}(x_n)}e^{2u_n}dx>\delta_0,\end{equation} when $n$ is big enough.

Denote $$\bar{u}_n(r):=\frac{1}{2\pi r}\int_{\partial B_r(x_n)}u_n(x)d\theta.$$ Then
\begin{align*}
\frac{d}{dr}\bar{u}_n(r)=\frac{1}{2\pi r}\int_{ B_r(x_n)}\Delta u_n(x) dx=-\frac{1}{2\pi r}\int_{ B_r(x_n)}(2e^{2u_n}-e^{u_n}|\psi_n|^2) dx
\end{align*} and $$\frac{d}{dr}\bar{u}_n(r)=-\frac{b+o(1)}{r},\ \ for\ \ r=s_n^1.$$
Furthermore, for any $r\in [s_n^1,s_n^2]$,  we obtain
\begin{align*}
r\frac{d}{dr}\bar{u}_n(r)&=s_n^1\frac{d}{dr}\bar{u}_n(s_n^1)-\frac{1}{2\pi}\int_{ B_r(x_n)\setminus B_{s_n^1}(x_n)}(2e^{2u_n}-e^{u_n}|\psi_n|^2) dx\\
&=s_n^1\frac{d}{dr}\bar{u}_n(s_n^1)-\frac{1}{2\pi}\int_{ B_r(x_n)\setminus B_{s_n^1}(x_n)}(2e^{2u_n}) dx+o(1)\\
&\leq -b+o(1)\leq -\frac{b+1}{2},
\end{align*} when $n$ is big enough, where the second equality follows from the spinor's energy identity.

This yields $$\bar{u}_n(r)\leq \bar{u}_n(s_n^1)-b\ln\frac{r}{s_n^1},\ \ \forall\ \ r\in [s_n^1,s^2_n].$$ Integrating above inequality over $[s_n^1,s_n^2]$, we get
\begin{align*}
\int_{B_{s_n^2}(x_n)\setminus B_{s_n^1}(x_n)}e^{2u_n(x)}dx\leq C\int_{B_{s^2_n}(x_n)\setminus B_{s_n^1}(x_n)}e^{2\bar{u}_n(r)}dx\leq \frac{C}{b-1}e^{2\bar{u}_n(s_n^1)}(s_n^1)^2=o(1),
\end{align*} where the first inequality follows from the second conclusion of Lemma \ref{lem:01}, the last equality follows from the fact that $u_n$ gas fast decay on $\partial B_{s_n^1}(x_n)$ which is derived from the first conclusion of Lemma \ref{lem:01}. This is a contradiction to \eqref{equat:26}. We proved the lemma.

\end{proof}

\

Now we prove Lemma \ref{lem:06}.
\begin{proof}[\textbf{Proof of Lemma \ref{lem:06}}]
Without loss of generality, we assume again that there are at most two bubbles at a considered blow-up point $p\in\Sigma\setminus \Sigma_3$.

We just show the proof for the case of $m(p)=4\pi$, since the other case $m(p)=0$ is similar and in fact easier.

By the process of constructing the first bubble, we know that there exist $x_n\to p$ and $\lambda_n\to 0$ such that   \begin{equation*}
\tilde{u}_n(x):=u_n (\lambda_n x+ x_n) + \log\lambda_n\to -\infty,\ \ \mbox{uniformly in any compact subset of }\R^2 \end{equation*} and $$\tilde{\psi}_n(x) := \lambda_n^{\frac{1}{2}} \psi_n (\lambda_n x + x_n) \to \tilde{\psi}\ \ in\ \ C^2_{loc}(\R^2)$$ where $\tilde{\psi}$ is a spinorial Yamabe type bubble.

Next, we just need to show that there must have a second bubble $(v,\phi)$ which is a super-Liouville type bubble. Moreover, it satisfies the following energy identity
\begin{equation}
\lim_{\delta\to 0}\lim_{n\to\infty}\int_{B_\delta(p)}e^{2u_n}dx=\int_{S^2}e^{2v}dx.
\end{equation}

\

We firstly have following two claims:

\

\textbf{Claim 1:} There must be a energy concentration in neck domain, i.e.
\begin{equation*}
\lim_{\delta\to 0} \lim_{R\to\infty} \lim_{n\to\infty}  \sup_{t\in [\lambda_nR,\delta]}\int_{B_{2t}(x_n)\setminus B_t(x_n)}( e^{2u_n} + |\psi_n|^4)dx \geq\epsilon_0>0.
\end{equation*}
Otherwise, we know that Lemma \ref{lem:01} and Lemma \ref{lem:02} hold.

Since $m(p)=4\pi$, then we have
\begin{align}\label{inequ:08}
4\pi&=\lim_{\delta\to 0}\lim_{n\to\infty}\int_{B_{\delta}(p)}(2e^{2u_n}-e^{u_n}|\psi_n|^2)dx\notag\\
&= \lim_{R\to\infty}\lim_{\delta\to 0}\lim_{n\to\infty}\int_{B_{\delta}(p)\setminus B_{\lambda_nR}(x_n)}(2e^{2u_n}-e^{u_n}|\psi_n|^2)dx+\lim_{R\to\infty}\lim_{\delta\to 0}\lim_{n\to\infty}\int_{ B_{\lambda_nR}(x_n)}(2e^{2u_n}-e^{u_n}|\psi_n|^2)dx\notag\\
&= \lim_{R\to\infty}\lim_{\delta\to 0}\lim_{n\to\infty} \int_{B_{\delta}(p)\setminus B_{\lambda_nR}(x_n)}(2e^{2u_n}-e^{u_n}|\psi_n|^2)dx.
\end{align}
Then there must exist $\lambda_nR<<\leq s_n^1<<\delta$ (the notation $s_n^1<<s_n^2$ means $\lim_{n\to\infty}\frac{s_n^2}{s_n^1}=+\infty$), such that $$\lim_{n\to\infty}\int_{B_{s_n^1}(p)}(2e^{2u_n}-e^{u_n}|\psi_n|^2)dx=3\pi.$$

By Lemma \ref{lem:02}, we get $$\lim_{n\to\infty}\int_{B_\delta (p)\setminus B_{s_n^1}(p)}(2e^{2u_n}-e^{u_n}|\psi_n|^2)dx=0.$$ This is a contradiction to \eqref{inequ:08} where we in fact have  $$\lim_{\delta\to 0}\lim_{n\to\infty}\int_{B_{\delta}(p)\setminus B_{s_n^1}(p)}(2e^{2u_n}-e^{u_n}|\psi_n|^2)dx=\pi.$$ Then we prove Claim $1$.

\

By the proof of Lemma \ref{lem:04}, there exist $\epsilon_0>0$, $\delta>0,\ R>100$ and sequence $t_n\in [\lambda_nR,\delta]$ such that \begin{equation*}
  \int_{B_{2t_n}(x_n)\setminus B_{t_n}(x_n)} (e^{2u_n} + |\psi_n|^4 )dx\ge \epsilon_0,\ \ \lim_{n\to\infty}\frac{t_n}{\lambda_nR}=\lim_{n\to\infty}\frac{\delta}{t_n}=+\infty.
\end{equation*}
Denote
$$(v_n(x), \phi_n(x)):=(u_n(x_n+t_nx)+\ln t_n, \sqrt{t_n}\psi_n(x_n+t_nx)).$$

\

\textbf{Claim 2:} The $Case\ 2 $ in Lemma \ref{lem:04} will not happen.

\

In fact, if not, then $v_n \rightarrow -\infty$ uniformly on all compact subset of $\mathbb{R}^2 \setminus \{0\}$, we have $\phi_n \rightarrow \phi$ in $C^2_{loc}(\mathbb{R}^2\setminus \{0\})$, where $\phi$ is a spinorial Yamabe type bubble which is the second bubble. It is easy to see that
\begin{equation}\label{equat:12}
\lim_{T\to\infty}\lim_{n\to\infty}\int_{B_{Tt_n}(x_n)\setminus B_{T^{-1}t_n}(x_n)}(2e^{2u_n}-e^{u_n}|\psi_n|^2)dx=0.\end{equation}
By \eqref{inequ:08} and \eqref{equat:12}, it is not hard to conclude  that there exists $s_n^1\in [\lambda_nR,\delta]$ such that $\lambda_nR<<s_n^1<<t_n$ or $t_n<<s_n^1<<\delta$ and $$\lim_{n\to\infty}\int_{B_{s_n^1}(p)}(2e^{2u_n}-e^{u_n}|\psi_n|^2)dx=3\pi.$$

Since there is no third bubble, we have
\begin{align*}
&\lim_{\delta\to 0} \lim_{T\to\infty} \lim_{n\to\infty}  \sup_{t\in [Tt_n,\delta]}\int_{B_{2t}(x_n)\setminus B_t(x_n)}( e^{2u_n} + |\psi_n|^4)dx \\&=\lim_{\delta\to 0} \lim_{T\to\infty} \lim_{n\to\infty}  \sup_{t\in [\lambda_nR,T^{-1}t_n]}\int_{B_{2t}(x_n)\setminus B_t(x_n)}( e^{2u_n} + |\psi_n|^4)dx=0.
\end{align*}
By Lemma \ref{lem:02} and \eqref{equat:12}, we get $$\lim_{n\to\infty}\int_{B_\delta (p)\setminus B_{s_n^1}(p)}(2e^{2u_n}-e^{u_n}|\psi_n|^2)dx=0.$$ This is a contradiction to \eqref{inequ:08}.

\

Now we have to consider the other two cases in Lemma \ref{lem:04}.

\

\textbf{Case 1 in Lemma \ref{lem:04}:} There exists $T>0$ such that $(v_n,\phi_n)$ has a energy concentration point $q\in  B_{T}(0)\setminus B_{T^{-1}}(0)$, i.e.
\begin{eqnarray}
  \lim_{n\rightarrow \infty}\int_{B_r(q)}( e^{2v_n} + |\phi_n|^4 )dx\ge \epsilon_0>0,\ \ \forall \ r>0.
\end{eqnarray}

In this case, we know that for sequence $(v_n,\phi_n)$, the blow-up point $q$ must be first type singularity. Otherwise, by the proof in \textbf{Case 1} of Lemma \ref{lem:05}, we get \eqref{equat:12}. By the proof of Claim $2$, this also implies a contradiction to \eqref{inequ:08}.

\

Then there exist $y_n\to q$, $\mu_n\to 0$ such that
\begin{align*}
(\tilde{v}_n(x),\tilde{\phi}_n(x)):=(v_n(y_n+\mu_nx)+\ln\mu_n,\sqrt{\mu_n}\phi_n(y_n+\mu_nx))\ \ \to (\tilde{v},\tilde{\phi})\ \ &in\ \ C^2_{loc}(\R^2),
\end{align*} where $(\tilde{v},\tilde{\phi})$ is a super-Liouville type bubble which is the second bubble.

Since there is no third bubble, by the discussion in \textbf{Case 1} of Lemma \ref{lem:05}, we have the following properties:

\begin{itemize}
\item[(1)]
 The energy identity
\begin{equation}
\lim_{R\to +\infty}\lim_{\delta\to 0}\lim_{n\to\infty}\int_{B_\delta(q)\setminus B_{\mu_nR}(y_n)}|\phi_n|^4dx=0.
\end{equation}

\item[(2)] The weak limit  of $(v_n,\phi_n)$ must be trivial, i.e. $v_n(x)\to -\infty$ uniformly in any compact subset of $\R^2\setminus \{0,q\}$ and $\phi_n(x)\to 0$ in $C^2_{loc}(\R^2\setminus \{0,q\})$.

\item[(3)] For any $T>100$, we have
\begin{equation}
\lim_{\delta\to 0}\lim_{n\to\infty}\int_{B_\delta(p)\setminus B_{Tt_n}(x_n)}|\psi_n|^4dx=\lim_{R\to +\infty}\lim_{n\to\infty}\int_{B_{T^{-1}t_n}(x_n)\setminus B_{\lambda_nR}(x_n)}|\psi_n|^4dx=0,
\end{equation} which implies $$\lim_{n\to\infty}\int_{B_{T^{-1}t_n}(x_n)}e^{2u_n}|\psi_n|^2dx=\lim_{n\to\infty}\int_{ B_{\lambda_nR}(x_n)}e^{2u_n}|\psi_n|^2dx=\lim_{n\to\infty}\int_{ B_{R}(0)}e^{2v_n}|\phi_n|^2dx=0.$$
\end{itemize}

We can conclude that
\begin{align*}
&\lim_{R\to\infty}\lim_{n\to\infty}\int_{\left(B_{Tt_n}(x_n)\setminus (B_{T^{-1}t_n}(x_n)\right)\setminus B_{t_n\mu_n R}(x_n+t_ny_n)}e^{u_n}|\psi_n|^2dx\\
&=\lim_{R\to\infty}\lim_{n\to\infty}\int_{\left(B_T(0)\setminus (B_{T^{-1}}(0)\right)\setminus B_{\mu_n R}(y_n)}e^{v_n}|\phi_n|^2dx\\
&=\lim_{\delta\to 0}\lim_{R\to\infty}\lim_{n\to\infty}\int_{\left(B_T(0)\setminus B_{T^{-1}}(0)\right)\setminus B_\delta (q)}e^{v_n}|\phi_n|^2dx-\lim_{\delta\to 0} \lim_{R\to\infty}\lim_{n\to\infty}\int_{B_\delta (q)\setminus B_{\mu_n R}(y_n)}e^{v_n}|\phi_n|^2dx=0.
\end{align*}

Thus,
\begin{align*}
&\lim_{\delta\to 0}\lim_{R\to\infty}\lim_{n\to\infty}\int_{B_{\delta}(p)\setminus B_{t_n\mu_n R}(x_n+t_ny_n)}e^{u_n}|\psi_n|^2dx\\
&=\lim_{\delta\to 0}\lim_{n\to\infty}\int_{B_\delta(p)\setminus B_{Tt_n}(x_n)}e^{u_n}|\psi_n|^2dx+ \lim_{R\to\infty}\lim_{n\to\infty}\int_{\left(B_{Tt_n}(x_n)\setminus (B_{T^{-1}t_n}(x_n)\right)\setminus B_{t_n\mu_n R}(x_n+t_ny_n)}e^{u_n}|\psi_n|^2dx\\
&\quad +\lim_{\delta\to 0}\lim_{n\to\infty}\int_{ B_{T^{-1}t_n}(x_n)}e^{u_n}|\psi_n|^2dx=0.
\end{align*}

Integrating by parts, we have
\begin{align*}
\int_{B_{\delta}(p)\setminus B_{t_n\mu_n R}(x_n+t_ny_n)}(2e^{2u_n}-e^{u_n}|\psi_n|^2)dx&=\int_{B_{\delta}(p)\setminus B_{t_n\mu_n R}(x_n+t_ny_n)}(-\Delta u_n)dx\notag\\
&=-\int_{\partial B_{\delta}(p)}\frac{\partial u_n}{\partial r}dx+\int_{\partial B_{ R}(0)}\frac{\partial \tilde{v}_n}{\partial r}dx =o(1),
\end{align*} where $\lim_{\delta\to 0}\lim_{R\to\infty}\lim_{n\to\infty} o(1)=0$.

By above two equalities, we get
\begin{align*}
&\lim_{\delta\to 0}\lim_{R\to\infty}\lim_{n\to\infty}\int_{B_{\delta}(p)}e^{2u_n}dx\\&=\lim_{\delta\to 0}\lim_{R\to\infty}\lim_{n\to\infty}\int_{B_{\delta}(p)\setminus B_{t_n\mu_n R}(x_n+t_ny_n)}e^{2u_n}dx+ \lim_{\delta\to 0}\lim_{R\to\infty}\lim_{n\to\infty}\int_{ B_{t_n\mu_n R}(x_n+t_ny_n)}e^{2u_n}dx=\int_{\R^2}e^{2\tilde{v}}dx.
\end{align*}

\

\textbf{Case 3 in Lemma \ref{lem:04}:} $(v_n,\phi_n)\to (v,\phi)$ in $C^2_{loc}(\R^2\setminus \{0\})$ and $(v,\phi)$ is a super-Liouville type bubble, which is the second bubble.

\

Since there is no third bubble, similar to the proof of Lemma \ref{lem:04}, then we  have
\begin{align*}
\lim_{T\to\infty}  \lim_{\delta\to 0}\lim_{n\to\infty}  \sup_{t\in [Tt_n,\delta]}\int_{B_{2t}(x_n)\setminus B_t(x_n)}( e^{2u_n} + |\psi_n|^4)dx&=0,\\ \lim_{T\to\infty}  \lim_{R\to +\infty}\lim_{n\to\infty}  \sup_{t\in [\lambda_nR,T^{-1}t_n]}\int_{B_{2t}(x_n)\setminus B_t(x_n)}( e^{2u_n} + |\psi_n|^4)dx &=0,
\end{align*} which implies the energy identity
\begin{equation}\label{equat:19}
\lim_{T\to\infty}\lim_{\delta\to 0}\lim_{n\to\infty}\int_{B_\delta(p)\setminus B_{Tt_n}(x_n)}|\psi_n|^4dx=\lim_{T\to\infty}\lim_{R\to +\infty}\lim_{n\to\infty}\int_{B_{T^{-1}t_n}(x_n)\setminus B_{\lambda_nR}(x_n)}|\psi_n|^4dx=0.
\end{equation}

On one hand, integrating by parts, we have
\begin{align*}
\int_{B_{\delta}(p)\setminus B_{t_n T}(x_n)}(2e^{2u_n}-e^{u_n}|\psi_n|^2)dx&=\int_{B_{\delta}(p)\setminus B_{t_n T}(x_n)}(-\Delta u_n)dx\notag\\
&=-\int_{\partial B_{\delta}(p)}\frac{\partial u_n}{\partial r}dx+\int_{\partial B_{ T}(0)}\frac{\partial v_n}{\partial r}dx =o(1),
\end{align*} where $\lim_{T\to\infty}\lim_{\delta\to 0}\lim_{n\to\infty} o(1)=0$, which implies (by using \eqref{equat:19}) $$\lim_{T\to\infty}\lim_{\delta\to 0}\lim_{n\to\infty}\int_{B_\delta(p)\setminus B_{Tt_n}(x_n)}e^{2u_n}dx=0.$$

On the other hand, by the proof in \textbf{Case 3} of Lemma \ref{lem:04}, we get
\begin{align*}
\lim_{T\to\infty}\lim_{n\to\infty}\int_{ B_{T^{-1}t_n}(x_n)}(2e^{2u_n}-e^{u_n}|\psi_n|^2)dx= \lim_{T\to\infty}\lim_{n\to\infty}\int_{ B_{T^{-1}}(0)}(2e^{2v_n}-e^{v_n}|\phi_n|^2)dx=\gamma=0,
\end{align*} which implies (by using \eqref{equat:19})
\begin{align*}
&\lim_{T\to\infty}\lim_{n\to\infty}\int_{ B_{T^{-1}t_n}(x_n)}2e^{2u_n}dx\\&= \lim_{T\to\infty}\lim_{n\to\infty}\int_{ B_{T^{-1}t_n}(x_n)}e^{u_n}|\psi_n|^2dx\\
&=\lim_{R\to\infty}\lim_{T\to\infty}\lim_{n\to\infty}\int_{ B_{T^{-1}t_n}(x_n)\setminus B_{\lambda_nR}(x_n)}e^{u_n}|\psi_n|^2dx+ \lim_{R\to\infty}\lim_{T\to\infty}\lim_{n\to\infty}\int_{ B_{\lambda_nR}(x_n)}e^{u_n}|\psi_n|^2dx=0.
\end{align*}
Thus,
\begin{align*}
\lim_{\delta\to 0}\lim_{n\to\infty}\int_{B_\delta(p)}e^{2u_n}dx&= \lim_{T\to\infty}\lim_{\delta\to 0}\lim_{n\to\infty}\int_{B_\delta(p)\setminus B_{Tt_n}(x_n)}e^{2u_n}dx+ \lim_{T\to\infty}\lim_{n\to\infty}\int_{ B_{T^{-1}t_n}(x_n)}e^{2u_n}dx\\
&\quad +\lim_{T\to\infty}\lim_{n\to\infty}\int_{B_{Tt_n}(x_n)\setminus B_{T^{-1}t_n}(x_n)}2e^{2u_n}dx\\
&=\lim_{T\to\infty}\lim_{n\to\infty}\int_{B_{T}(0)\setminus B_{T^{-1}}(0)}2e^{2v_n}dx=\int_{\R^2}e^{2v}dx.
\end{align*}
We proved the lemma.
\end{proof}

\

\begin{proof}[\textbf{Proof of Theorem \ref{thm:main-5}:}] It is easy to see that the conclusion of Theorem \ref{thm:main-5} is a consequence of Lemma \ref{lem:05} and Lemma \ref{lem:06}.
\end{proof}

\

\begin{proof}[\textbf{Proof of Theorem \ref{thm:main-6}:}]
By Theorem \ref{thm:main-1}, we have following two cases corresponding to two conclusions of Theorem \ref{thm:main-6}.

\

\textbf{Case 1:} $u_n$ is uniformly bounded in $L^\infty_{loc}(M\setminus \Sigma)$.

\

In this case,  passing to a subsequence, we have $(u_n,\psi_n)\to (u,\psi)$ in $C^2_{loc}(M\setminus \Sigma)$ where $(u,\psi)$ satisfies
 \begin{equation*}
    \begin{cases}
          -\Delta u &= 2e^{2u} - e^u |\psi|^2, \\
        \slashed{D} \psi &= -(e^u + 2F|\psi|^2) \psi,
    \end{cases}\ \ in \ \ M\setminus \Sigma.
  \end{equation*}
By Proposition \ref{Removability of local singularity} and the proof of Lemma \ref{lem:04}, for each $p\in\Sigma$, we know that $$u(x)=-\frac{\gamma_p}{2\pi}\ln |x-p|+O(1)\ \ near\ \ p,$$  where $\gamma<2\pi$ is a constant defined by $$\gamma_p=\lim_{\delta\rightarrow 0}\lim_{n\rightarrow \infty} \int_{B_\delta(p) } \left(2e^{2u_n} - e^{u_n} |\psi_n|^2 \right)dx=0.$$  Combining with the removable singularity theorem for spinor's equation (see Theorem 5.1 in \cite{Jost-Liu-Zhu}), we get that $(u,\psi)$ is a smooth solution in $M$, which implies the first conclusion $(i)$ of $(1)$. The second conclusion $(ii)$ follows from that $m(p)=\gamma_p=0$ and the third conclusion follows from a similar  proof for the second conclusion of  Lemma \ref{lem:06}.

\

\textbf{Case 2:} Passing to a subsequence, $u_n\to -\infty$  uniformly on any compact subset of $M\setminus \Sigma$.

\

In this case, one can check that all the arguments in the proof of Lemma \ref{lem:Green funciton for u_n}, Theorem \ref{thm:main-4} and Theorem \ref{thm:main-5} still hold. Then the conclusions of $(2)$ follows from the proof of Theorem \ref{thm:main-4} and Theorem \ref{thm:main-5}.

We finished the proof.
\end{proof}


\begin{thebibliography}{10}

\bibitem{Brezis}
H. Brezis and F. Merle,  \emph{Uniform estimate and blow up behaviour for solutions of $-\Delta u=V(x)e^u$ in two dimensions}, Comm. Partial Differential Equations 16 (1991) 1223-1253.


\bibitem{Chen-Jost-Wang-1}
Q. Chen, J. Jost and G. Wang, \emph{Liouville theorems for Dirac-harmonic maps}, J. Math. Phys. 48 (2007), no. 11, 113517, 13pp.

\bibitem{Chen-Jost-Wang-2}
Q. Chen, J. Jost and G. Wang, \emph{Nonlinear Dirac equations on Riemann surfaces}, Ann. Global Anal. Geom. 33 (2008), no. 3, 253-270


\bibitem{Chen-Liu-Zhu}
Y. Chen, L. Liu and M. Zhu, \emph{Bubbling analysis for a nonlinear Dirac equation on surfaces}, Acta Math. Sci. Ser. B (Engl. Ed.) 44 (2024), no. 6, 2073-2082.

\bibitem{DingWeiyueandTiangang}
W. Ding and G. Tian, \emph{Energy identity for a class of approximate harmonic maps from surfaces}, Comm. Anal. Geom. \textbf{3} (1995), no.~3-4,  543--554.

\bibitem{Li}
Y. Li,  \emph{ Harnack type inequality: the method of moving planes}, Commun. Math. Phys. 200(2), 421-444 (1999).

\bibitem{Li-Shafrir}
Y. Li and I. Shafrir, \emph{Blow-up analysis for solutions of $- \Delta u=Ve^u$ in dimension two,} Indiana Univ. Math. J., \textbf{43} (1994), 1255-1270.

\bibitem{Jost-Liu-Zhu}
J. Jost, L. Liu and M. Zhu, \emph{Geometric analysis of the action functional of the nonlinear supersymmetric sigma model},  Calc. Var. Partial Differential Equations 61 (2022), no. 3, Paper No. 112, 26 pp

\bibitem{Jost-Wang-Zhou}
J. Jost, G. Wang and C. Zhou, \emph{Super-Liouville equations on closed Riemann surfaces}, Comm. Partial Differential Equations \textbf{32} (2007), no. 7-9, 1103-1128.

\bibitem{Jost-Wang-Zhou-Zhu-1}
J. Jost, G. Wang, C. Zhou and M. Zhu, \emph{Energy identities and blow-up analysis for solutions of the super Liouville equation},  J. Math. Pures Appl. (9) 92 (2009), no. 3, 295-312.

\bibitem{Jost-Zhou-Zhu}
J. Jost,  C. Zhou and M. Zhu, \emph{Vanishing Pohozaev constant and removability of singularities},  J. Differential Geom. 111 (2019), no. 1, 91-144.

\bibitem{lawson1989spin}
H. B. Lawson and M. L. Michelsohn, \emph{Spin geometry}, vol.~38, Princeton University Press, 1989.


\bibitem{Polyakov}
A.M. Polyakov, \emph{Quantum geometry of fermionic strings}, Phys. Lett. B \textbf{103} (1981) 207-210.

\bibitem{Taimanov}
I. A. Taimanov, \emph{The two-dimensional Dirac operator and the theory of surfaces} (Russian, with Russian summary), Uspekhi Mat. Nauk 61 (2006), no. 1(367), 85-164; English transl., Russian Math. Surveys 61 (2006), no. 1, 79-159.

\bibitem{zhao2007energy}
L. Zhao, \emph{Energy identities for Dirac-harmonic maps}, Calc. Var. Partial Differential Equations \textbf{28} (2007), no.~1,
  121-138.

\bibitem{Zhu}
M. Zhu, \emph{Quantization for a nonlinear Dirac equation}, Proc. Amer. Math. Soc. 144 (2016), no. 10, 4533-4544.


\end{thebibliography}

\providecommand{\bysame}{\leavevmode\hbox to3em{\hrulefill}\thinspace}
\providecommand{\MR}{\relax\ifhmode\unskip\space\fi MR }
\providecommand{\MRhref}[2]{%
  \href{http://www.ams.org/mathscinet-getitem?mr=#1}{#2}
}
\providecommand{\href}[2]{#2}

\end{document}